\documentclass[a4paper,12pt,reqno,twoside]{amsart}

\usepackage[utf8]{inputenc} 		% UTF8 encoding
\usepackage[T1]{fontenc} 

\usepackage{amssymb}
\usepackage{amsmath}
\usepackage{graphicx}
\usepackage[colorlinks,citecolor=blue,urlcolor=blue]{hyperref}
\usepackage{cite}
\usepackage[hmargin=2.5cm,vmargin=2.5cm]{geometry}
\usepackage{mathrsfs} % Script fonts
\usepackage{ellipsis}
\usepackage{graphicx}
\usepackage{svg}

\usepackage{dirtytalk}

\newcommand{\be}{\begin{eqnarray}}
\newcommand{\ee}{\end{eqnarray}}

\newcommand{\beq}{\begin{equation}}
\newcommand{\eeq}{\end{equation}}
\newcommand{\beqn}{\begin{equation*}}
\newcommand{\eeqn}{\end{equation*}}

\newcommand{\round}[1]{\lfloor#1\rfloor}

\newtheorem{thm}{Theorem}[section]

\newtheorem{lem}[thm]{Lemma}
\newtheorem{defn}[thm]{Definition}

\theoremstyle{remark}
\newtheorem{claim}[thm]{Claim}
\newtheorem{remark}[thm]{Remark}

\newcommand\cB{{\mathcal B}}
\newcommand\cC{{\mathcal C}}

\newcommand\cL{{\mathcal L}}

\newcommand\cN{{\mathcal N}}

\newcommand\cP{{\mathcal P}}

\newcommand\cR{{\mathcal R}}

\newcommand\cT{{\mathcal T}}

\newcommand\bR{{\mathbb R}}

\newcommand\fD{{\mathfrak D}}

\newcommand{\ve}{\varepsilon}

\begin{document}

\title[Linear response for
intermittent maps with critical point]{
  Linear response for \\
intermittent maps with critical point}

%\author[Olli Hella]{Olli Hella}
%\address[Olli Hella]{
%Department of Mathematics and Statistics, P.O.\ Box 68, Fin-00014 University of Helsinki, Finland.}
%\email{olli.hella@helsinki.fi}
%\urladdr{http://www.math.helsinki.fi/mathphys/mikko.html}

\author[Juho Lepp\"anen]{Juho Lepp\"anen}
%\address[Juho Lepp\"anen]{
%LPSM, Laboratoire de Probabilit\'es, Statistique et 
%Mod\'elisation, 
%Sorbonne Universit\'e, 4 Place Jussieu, 75252 Paris, France}
%\email{leppanen@lpsm.paris}
%\urladdr{http://www.math.helsinki.fi/mathphys/mikko.html}

%\footnote{ss}
\thanks{\textsc{Department of Mathematics, 
   Tokai University, Kanagawa, 259-1292, Japan}}
\thanks{ \textit{E-mail address}: leppanen@tsc.u-tokai.ac.jp}
\thanks{\textit{Date}:  \today }
%\keywords{linear response, critical 
%point, intermittent map}
 % Suggesting these. 
\thanks{This research was supported by 
JSPS via the project LEADER.
The author is grateful to Alexey Korepanov, 
Yushi Nakano, and 
Mikko Stenlund for useful 
comments and suggestions,
and to the anonymous referees whose comments enhanced the quality of the manuscript, leading to improvements in the presentation and the main result.}
 
% http://www.ams.org/msc/msc2010.html

% 37A05  	Measure-preserving transformations
% 37A35  	Entropy and other invariants, isomorphism, classification
% 37A40  	Nonsingular (and infinite-measure preserving) transformations
% 37A45  	Relations with number theory and harmonic analysis [See also 11Kxx]
% 37A10  	One-parameter continuous families of measure-preserving transformations
% 37A15  	General groups of measure-preserving transformations [See mainly 22Fxx]
% 37A17  	Homogeneous flows [See also 22Fxx]
% 37A20  	Orbit equivalence, cocycles, ergodic equivalence relations
% 37A25  	Ergodicity, mixing, rates of mixing
% 37A30  	Ergodic theorems, spectral theory, Markov operators
% 37A50  	Relations with probability theory and stochastic processes [See also 60Fxx and 60G10]

% 37D20  	Uniformly hyperbolic systems (expanding, Anosov, Axiom A, etc.)
% 37D25  	Nonuniformly hyperbolic systems (Lyapunov exponents, Pesin theory, etc.)
% 37C40  	Smooth ergodic theory, invariant measures
% 37C60  	Nonautonomous smooth dynamical systems

% 60F05  	Central limit and other weak theorems
% 60G44  Martingales with continuous parameter
% 60H10  	Stochastic ordinary differential equations

%\date{\today. {\bf Please do not circulate!}}

\begin{abstract} We consider a two-parameter 
    family of maps $T_{\alpha, \beta} 
    : [0,1] \to [0,1]$ with a neutral fixed 
    point and a non-flat critical point.
Building on a cone technique due to Baladi and Todd, 
we show that for a class of $L^q$ observables $\phi
: [0,1] \to \bR$
the bivariate map $(\alpha, \beta) \mapsto \int_0^1 
\phi \, d\mu_{\alpha,\beta}$, where 
$\mu_{\alpha, \beta}$ denotes the invariant
SRB measure, is differentiable 
in a certain parameter region, and establish 
a formula for its directional derivative. 
\end{abstract}

\maketitle

%%%%%%%%%%%%%%%%%%%%%%%%%%%%
%%%%%%%%%%%%%%%%%%%%%%%%%%%%

%%%%%%%%%%%%%%%%%%%%%%%%%%%%%%%%%%%%
%%%%%%%%%%%%%%%%%%%%%%%%%%%%%%%%%%%%
%%%%%%%%%%%%%%%%%%%%%%%%%%%%%%%%%%%%

\section{Introduction}\label{sec:intro}

For $0 < \alpha 
< \tfrac{1}{\beta} \le 1$, 
the following two-parameter map 
$T_{\alpha, \beta} : [0,1] \to [0,1]$ was
considered by Cui \cite{cui2021invariant}:
\begin{align}\label{eq:cui_map}
    T_{\alpha, \beta}(x)
    = \begin{cases}
      x( 1 + 2^{\alpha} x^{\alpha} ) \quad x \in [0, \tfrac12), \\
      2^{\beta}( x - \tfrac12 )^{\beta}, 
      \quad x \in [\tfrac12, 1].
    \end{cases}
  \end{align}
The map has a neutral (or indifferent) 
fixed point at $x = 0$ 
and a non-flat critical point at $x = \tfrac12$, 
which strengthens the effect of the neutral 
fixed point on essential properties of the dynamics.
By \cite{cui2021invariant,inoue1992}, piecewise 
convex 
interval maps such as $T_{\alpha, \beta}$ 
have a unique absolutely continuous 
invariant probability measure $\mu_{\alpha, \beta}$
with density $h_{\alpha, \beta}(x)
\lesssim x^{ -1 + \frac{1}{\beta} - \alpha }$. 
We are going to study the problem of linear 
response for the family \eqref{eq:cui_map}, 
i.e. 
the differentiability of the map 
$(\alpha, \beta) \mapsto \int_0^1 \phi \, d \mu_{
    \alpha, \beta
}$, where $\phi \in L^q$ with $q$ sufficiently 
large.

Invariant measures play an important role
in describing the long-term 
statistical behavior
of dynamical systems, and understanding 
the robustness of these measures to 
perturbations is of both theoretical and practical interest.
Ruelle \cite{ruelle1998differentiation, 
ruelle2009review} 
showed that perturbations of uniformly hyperbolic Axiom A dynamical systems result in smooth changes to the associated SRB measure.
In this case, the 
derivative of the SRB measure can be 
expressed by a formula -- 
the linear response formula -- which 
explicitly 
describes the rate of change of the 
measure in response to the perturbation
based on statistical information 
from the unperturbed system.
Such ideas have found
applications outside mathematics, 
for example in the areas of climate 
science 
\cite{ragone2016new,
lucarini2011statistical
}
and neurophysiology
\cite{herfurth2017linear}. On the other hand, 
mathematical investigations of linear response have resulted in extensions of Ruelle's results to 
various scenarios where the system and its perturbations are sufficiently smooth and hyperbolic.
In 
\cite{dolgopyat2004differentiability, 
zhang2018smooth}, 
linear response 
formulas were established 
for perturbations of 
partially hyperbolic systems.
There has been deep investigation about 
response  
in the setting of unimodal maps 
\cite{baladi2021fractional,
baladi2007susceptibility,
sedro2021prethreshold,
aspenberg2022fractional,
ruelle2009structure
}, 
providing examples of systems where the  
SRB measure does not vary smoothly 
under perturbations.
Recently, linear response 
has also been actively studied in 
the context of random and extended dynamical 
systems; see, for instance  
\cite{crimmins2021spectral,
selley2021linear, koltai2019frechet, 
galatolo2019linear, 
dragicevic2023quenched}. For extensive background information on linear response, 
we recommend the survey \cite{baladi2014linear}.

Returning to the example \eqref{eq:cui_map}, 
this map can be viewed as 
an extension of the intermittent map 
popularized by Liverani, Saussol and Vaienti 
\cite{liverani1999probabilistic}, 
which is 
obtained 
from \eqref{eq:cui_map} by 
taking $\beta = 1$. 
Similar maps were studied numerically by 
Pomeau and Manneville \cite{pomeau1980intermittent} 
as models of 
intermittent
phenomena, such as
transitions to turbulence in convective fluids.
The map $T_{\alpha, 1}$ 
is uniformly expanding except at 
the neutral fixed point at the origin.
In this 
case, robustness of statistical properties 
have been investigated for both deterministic
\cite{korepanov2016linear,
baladi2016linear,selley2022differentiability,
bahsoun2016linear, freitas2009statistical, 
nisoli2022rigorous}
and random perturbations \cite{bahsoun2020linear}.
Independently, Baladi and Todd \cite{baladi2016linear} and Korepanov \cite{korepanov2016linear} analyzed the regularity properties of the map
$
\cR_{\phi}(\alpha) := \int_0^1 \phi \, d \mu_{\alpha,1}
$ for $\phi \in L^q$ with $q > 1$ sufficiently large, 
employing different methods.
In 
\cite{baladi2016linear}, a cone method was 
used to derive 
a linear response formula for the derivative of 
$
\cR_{\phi}(\alpha)$, 
while  \cite{korepanov2016linear} developed a 
a mechanism through inducing and coupling  
to show that $\cR_\phi$ is $C^2$ on compact subsets 
of $(0,1)$. The linear response 
formula of \cite{baladi2016linear} 
describes the derivative of $\cR_\phi(\alpha)$ 
in terms of the unperturbed system 
$T_{\alpha, 1}$ and $\phi$, together with $\partial_{\alpha}
T_{\alpha, 1}
$ and the SRB measure $\mu_{\alpha,1}$. In 
the present paper, we give a similar description 
of the derivative in the case of the two-parameter 
family \eqref{eq:cui_map}.

\begin{figure}[h!]
  \centering
      \includegraphics[width=1.0\textwidth]{
        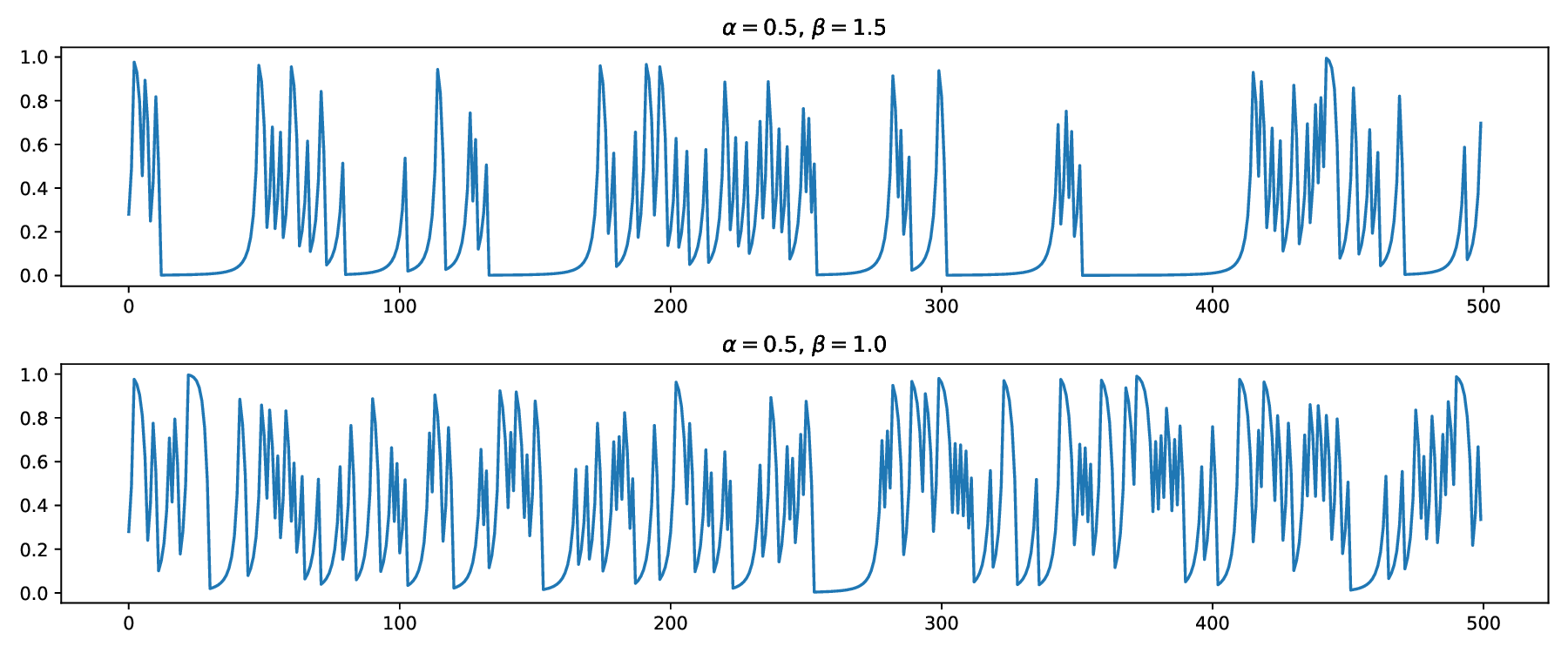
      }
  \caption{Intermittent behavior of 
  trajectories 
  $T_{\alpha,\beta}^n(x)$  
  for $(\alpha, \beta) = (0.5, 
  1.5)$ and $(\alpha, \beta) = (0.5, 
  1.0)$. The 
  latter case corresponds to 
  the standard Liverani--Saussol--Vaienti map 
  with parameter $\alpha = 0.5$. Once 
  the trajectory lands 
  near $\tfrac12$ in $(\tfrac12, 1]$, after 
  one iterate it will be in a small 
  neighborhood of $0$ and the larger 
  the 
  $\alpha\beta$ the longer it will take 
  for the trajectory to return to the 
  strongly chaotic region.}
\end{figure}

Statistical properties of piecewise convex 
interval maps, which admit both indifferent fixed points and critical points, have been previously 
investigated in \cite{cui2021invariant, 
coates2022doubly, inoue1992asymptotic,
inoue2023invariant, inoue1992
}. For the map \eqref{eq:cui_map}, 
Cui \cite{cui2021invariant} showed that 
correlations associated to $C^1$ observables 
decay at the rate $O(n^{ \frac{1}{\beta}( 
    1 - \frac{1}{\alpha\beta}
) } ( \log n )^{ 1 - \frac{1}{\beta}( 1 
- \frac{1}{\alpha\beta}) })$ based on 
the approach of Liverani, Saussol, and 
Vaienti \cite{liverani1999probabilistic}.
A sharper 
estimate of order $O(n^{  
    1 - \frac{1}{\alpha\beta}
 })$ for H\"{o}lder observables follows 
 by a recent result 
 due to Coates, Luzzatto, and 
 Muhammad \cite[Theorem D]{coates2022doubly}. 
 Their approach, based on Young's tower method 
 \cite{young1999recurrence}, applies to a wide range of full-branch 
 maps, and also yields advanced statistical 
 properties such as stable limit laws associated 
 to appropriately scaled Birkhoff sums. As 
 part of our proof of the linear response formula, 
 we derive the same rate $O(n^{  
    1 - \frac{1}{\alpha\beta}
 })$ for the map \eqref{eq:cui_map} by applying 
 a result due to Korepanov and 
 Lepp\"{a}nen \cite{korepanov2021loss}. 
 This application could be extended to 
 treat (random or deterministic) 
 time-dependent compositions 
 $T_{n} \circ 
 \cdots \circ T_{1}$ 
 of maps $T_{i}$ such as  
 \eqref{eq:cui_map}. In the recent work \cite{bahsoun2023mean}, the same 
 approach was used to establish decorrelation bounds in the context of 
 mean field coupled intermittent maps.

\subsection*{Notation and conventions}
In what follows, $C, 
C_0, C_1,\ldots$ denote 
positive constants whose 
values may chance from one 
line to the next. We use $C(a,b,\ldots)$ 
to indicate that the constant 
$C$ depends only on the parameters 
$a, b, \ldots$. Some other commonly used notation 
is listed below.

\begin{itemize}
  \item $L^q(\mu) := L^q([0,1], \cB, 
  \mu)$, where $\cB$ is the Borel sigma algebra 
  of $[0,1]$;\smallskip
  \item $m$ denotes the Lebesgue measure on $[0,1]$;\smallskip
  \item $\mu(f) := \int_0^1 f \, d\mu$ 
  for a measure $\mu$ 
  on $\cB$ and an integrable 
  function $f: [0,1] \to \bR$;\smallskip 
  \item $()'$ denotes differentiation with 
  respect to a spatial variable $x \in [0,1]$, 
  while $\partial_{i}$ for $i=1,2$ 
  denotes differentiation 
  with respect to the $i$th component of the 
  parameter vector $\gamma = (\alpha, \beta)$; \smallskip 
  \item $T_{\gamma} := T_{\alpha, \beta}$ for $\gamma = (\alpha, \beta)$ where 
  $0 < \alpha 
  < \tfrac{1}{\beta} \le 1$;
  \item $\cR_\phi(\gamma) = \cR_\phi(\alpha, \beta) := \int_0^1 \phi \, d \mu_\gamma$ where 
  $\phi \in L^1(\mu_\gamma)$, 
  $0 < \alpha 
  < \tfrac{1}{\beta} \le 1$, and $\mu_\gamma = \mu_{\alpha, \beta}$ denotes 
  the absolutely continuous invariant probability measure of $T_\gamma$. For $\beta < 1$ 
  we set $\cR_\phi(\alpha, \beta) = \cR_\phi(\alpha, 1)$.
\end{itemize}

%\section{Main result}

\subsection{Setting and main result} 
For
$$ 
\gamma = (\alpha, \beta) \in 
\mathfrak{D} := \{  (\alpha, \beta) \in 
(0, 1) \times [1, \infty) \: : \: \alpha \beta < 1 \},
$$
let $f_{\gamma, 1} : [0, \tfrac12] \to [0,1]$
and $f_{\gamma, 
2} : [\tfrac12, 1] \to [0,1]$ denote 
respectively the left and 
right branch  
of $T_{\gamma}$:
\begin{align*}
  f_{\gamma, 1}(x) = x( 1 + 2^{\alpha} x^{\alpha} ),
  \qquad f_{\gamma, 2}(x) = 2^{\beta}( x - \tfrac12 )^{\beta}.
\end{align*}
Let
$g_{\gamma, 1} : [0, 1] \to [0,1/2]$
and $g_{\gamma, 2} : [0, 1] \to [1/2,1]$ denote their inverses: 
\begin{align*}
  g_{\gamma, 1} = f_{\gamma, 1}^{-1} \quad \text{and}
  \quad g_{\gamma, 2} = f_{\gamma, 2}^{-1}.
\end{align*} 
Note that 
\begin{align*}
  g_{\gamma, 2}(x) = \tfrac{1}{2} ( x^{\frac{1}{\beta}} + 1 ).
\end{align*}
Moreover, $g_{\gamma, 1}(x) \in [\tfrac{x}{2}, x]$ and 
\begin{align*}
  | g_{\gamma, 1}(x) - x(1 - 2^{\alpha} x^{\alpha} ) |
  \le 2^{2\alpha} x^{2\alpha + 1} \quad \forall x \in [0,1].
\end{align*}
Next, for $i = 1,2$ and $x \in [0,1]$, define $v_{\gamma, i}(x) = 
\partial_i T_{\gamma}(x)$ where $\partial_i T_{\gamma}(x)$ 
denotes the partial derivative of $T_\gamma(x)$ with respect to the  $i$th component of 
the parameter vector 
$\gamma = (\alpha, \beta)$. Then,
\begin{align*}
  v_{\gamma, 1}(x) = 2^{\alpha}x^{1 + \alpha} (\log 2 + \log x) \quad \forall x \in (0,\tfrac12]
\end{align*}
and
\begin{align*}
  v_{\gamma, 2}(x) = (\log 2 + \log( x - \tfrac12 )) g_{\beta,2}(x) \quad \forall x \in (\tfrac12, 1].
\end{align*} 
Finally, set $X_{\gamma,i} = 
v_{\gamma, i} \circ g_{\gamma, i}$.
Then, 
\begin{align*}
  X_{\gamma, 1}(x)
  &= 2^{\alpha}g_{\gamma, 1} (x)^{1 + \alpha} 
  (\log 2 + \log g_{\gamma, 1} (x))  \quad \forall x \in (0,1]
\end{align*}
and 
\begin{align*}
  X_{\gamma, 2}(x) = (\log 2 + \log( \tfrac12 x^{\frac{1}{\beta}} 
  ) ) x
  = \tfrac{1}{\beta} \log( x ) x \quad \forall x \in (0, 1].
\end{align*}

For $\varphi \in L^1( m)$ and $i \in \{1,2\}$, denote 
\begin{align*}
  \cN_{\gamma, i}(\varphi)(x) = g_{\gamma, i}'(x) 
  \cdot \varphi(g_{\gamma, 1}(x)),
\end{align*}
so that the transfer operator $\cL_{\gamma}$
associated to $T_{\gamma}$ and $m$ 
satisfies 
\begin{align*}
  \cL_{\gamma}(\varphi)  = \cN_{\gamma, 1}(\varphi)
  + \cN_{\gamma, 2}(\varphi).
\end{align*}

By \cite[Theorem 1.1]{cui2021invariant}, 
there exists a unique absolutely continuous 
$T_{\gamma}$-invariant probability 
measure $\mu_{\gamma} = \mu_{\alpha,\beta}$ whenever 
$\gamma = (\alpha, \beta) \in \mathfrak{D}$.
For $\phi \in L^1(\mu_\gamma)$ 
define $\cR_\phi : \mathfrak{D} \to \bR$ by 
$$
\cR_\phi(\gamma) = \int_0^1 \phi \,
d\mu_{\gamma}.
$$
We adopt 
the 
convention $\cR_\phi(\alpha, \beta) 
= \cR_\phi(\alpha, 1)$ if 
$\beta < 1$. 

With the above preparations, the main result in this 
paper can be stated as follows:

\begin{thm}[Linear response]\label{thm:main} Let
  $
  \gamma = (\alpha, \beta) \in \mathfrak{D}.
  $
    Suppose that 
    $\phi : [0,1] \to \bR$ is a measurable 
    function such that 
    \begin{align}\label{eq:lq_cond}
    \int_0^1 | \phi(x) |^q x^{\frac{1}{\beta} - 
    \alpha
    - 1} \, dx < \infty
    \end{align}
    holds for some $q > 
    \tfrac{1 }{
      1 - \alpha \beta}$. Then, for $i \in \{1,2\}$,
      \begin{align}\label{eq:claim_well_defined}
      \sum_{k=0}^{\infty} \int_0^1
      \phi \cdot \cL_{\gamma}^k[ ( 
        X_{\gamma, i} \cN_{\gamma,i}
      (h_{\gamma})  )'  ] \, dm
      \end{align}
      is absolutely summable. Moreover, 
$\cR_\phi$
  is differentiable at $\gamma$
  with directional derivative  
  \begin{align*}
    \nabla_{-v} \cR_\phi(\gamma) &:= \lim_{\delta\to 0}\frac{ 
      \int_0^1 \phi \, d\mu_{\gamma - \delta v}
      - \int_0^1 \phi \, d\mu_{\gamma } }{\delta}\\
    &=
    \sum_{k=0}^{\infty} v_1 \int_0^1
    \phi \cdot \cL_{\gamma}^k[ ( X_{\gamma, 1} 
    \cN_{\gamma,1}
    (h_{\gamma})  )'  ] \, dm\\
    &+ \sum_{k=0}^{\infty} v_2 \int_0^1
    \phi \cdot \cL_{\gamma}^k[ ( 
      X_{\gamma, 2} \cN_{\gamma,2}
    (h_{\gamma})  )'  ] \, dm,
    \end{align*}
    for any unit vector $v = (v_1, v_2) \in \bR^2$.
\end{thm}

\begin{remark}
If $\phi \in L^\infty(m)$, then 
	\eqref{eq:lq_cond} holds for all $q \ge 1$ 
	so 
	that $\cR_\phi$ is 
	differentiable on $\mathfrak{D}$.
	The first partial derivative 
	$\partial_1 \cR(\gamma)$ coincides with the 
	linear response formula established in 
	\cite{baladi2016linear}.
\end{remark}

\begin{remark} For a wide class of regular observables 
$\varphi_1, \varphi_2$ with $m(\varphi_1) = m(\varphi_2)$ we have 
$$
\Vert \cL_\gamma^n ( \varphi_1 - \varphi_2 ) \Vert_{L^1(m)} = O(  n^{1 - \frac{1}{\alpha\beta}} ) 
\quad \text{as $n \to \infty$},
$$
as established, for instance, by Lemma \ref{lem:memory_loss} in this paper. For $\varphi_1 = \mathbf{1}$ 
and $\varphi_2 = h_\gamma$ the rate is optimal if $\beta = 1$, and we guess 
its optimality for $\beta > 1$ as well. The rate is summable only if $\alpha\beta < 2$. However, the 
function $( X_{\gamma, i} 
\cN_{\gamma,i}
(h_{\gamma})  )'(x)$ appearing in \eqref{eq:claim_well_defined} grows
more slowly than $h_\gamma(x)$ as $x$ tends to $0$, 
which will be used 
to establish $ \sum_{k=0}^\infty \Vert \cL_{\gamma}^k[ ( 
  X_{\gamma, i} \cN_{\gamma,2}
(h_{\gamma})  )' \Vert_{L^1(m)} < \infty$ through  
an application of \cite[Theorem 3.8]{korepanov2021loss}.
\end{remark}

\begin{remark}
If $\gamma = (0, \beta)$ with $\beta \ge 1$, it can be seen by the techniques of \cite{cui2021invariant} that the map $T_{\gamma} = T_{0, \beta}$ 
still has an invariant measure $\mu_\gamma$ with density $h_\gamma$ belonging to the cone 
$\cC_a(\gamma) = \cC_a(0, \beta)$, defined in Section 
\ref{sec:memory_loss}.  
In this scenario, the linear response formula continues to hold, but the proof in Section \ref{sec:proof_main} necessitates modifications. We provide a concise outline of these necessary adjustments.
For
$\gamma = (0, \beta)$, the tail bounds in Lemmas 2.4 and 2.8 decay exponentially, 
implying an exponential decay rate of memory loss in \eqref{eq:ml_weak} through \cite[Theorem 3.8]{korepanov2021loss}.
Moreover, for such $\gamma$ the upper bound \eqref{eq:dist_first} decays exponentially, and the results in Section 
\ref{sec:cone} continue to hold.
Then, 
the main issue in repeating the steps in the proof of Theorem \ref{thm:main} is 
that,
for $\alpha > 0$, the constant in \eqref{eq:ml_strong} becomes large as $\alpha \searrow 0$. 
However, similar to \cite{baladi2016linear}, this issue can be circumvented  
by writing 
$\sum_{j=0}^{k-1} \cL_{\gamma+\delta}^j ( 
\cL_{\gamma + \delta} - \cL_\gamma )
\cL_{\gamma}^{k-j-1} ( \mathbf{1} ) = - \sum_{j=0}^{k-1} \cL_{\gamma}^j ( 
 \cL_\gamma  - \cL_{\gamma + \delta} )
\cL_{\gamma + \delta}^{k-j-1} ( \mathbf{1} )$ in \eqref{eq:decomp_diff_int}. This decomposition 
together with the previously mentioned properties can be employed to derive the linear response formula.
\end{remark}

The proof of Theorem \ref{thm:main} is based 
on the approach of \cite{baladi2016linear}, 
but there are a number of 
subtle differences, 
three 
of which are 
highlighted below:

\begin{itemize}
  \item[(1)] We rely on results from 
  \cite{korepanov2021loss} to obtain 
  polynomially decaying 
  upper bounds on 
  $\Vert \cL_{\gamma}^n( \phi - \psi ) 
  \Vert_{L^1(m)}$ that are uniform in 
  $\gamma \in K$ for compact $K \subset \mathfrak{D}$, 
  provided that 
  $\phi$ and 
  $\psi$ are sufficiently regular densities. 
  We extend these $L^1$ bounds to $L^q$ bounds 
  using a change of measure combined with interpolation, employing techniques in the spirit of 
   \cite{nicol2021large, bunimovich2023maximal, su2022vector}.
   The $L^q$ bounds are used to obtain the linear response formula for  
  $\phi$ that satisfies \eqref{eq:lq_cond}. 
  \smallskip 
  \item[(2)] We avoid  
  using the cone $\cC_{*, 1}$ constructed in 
  \cite{baladi2016linear} and lower 
  bounds on $h_\gamma$ 
  by decomposing regular functions 
  into linear combinations of functions 
  in the cone from \cite{cui2021invariant} 
  (see the definition of $\cC_a(\gamma)$ in 
  Section \ref{sec:memory_loss}).
 \smallskip
 \item[(3)] For $\beta = 1$, the linear 
 growth of  
 $\int_{b_\ell}^1  | h'_{\alpha, 1}(x) | \, dx 
 \lesssim \ell$ combined with 
 distortion estimates were used in 
 \cite{baladi2016linear} 
 to establish 
 \begin{align}\label{eq:issue_estimate}
  \lim_{k\to\infty} \Vert \mathbf{1}_{ 
    \{x > b_\ell\}} 
 [ \cL_{\alpha,1}^k( \mathbf{1} - h_{\alpha, 1} ) ]'
 \Vert_{L^1(m)} = 0,  
 \end{align}
 where $b_\ell$ 
 denotes the preimage of $\tfrac12$ under
 $(T_{\alpha, \beta} |_{[0, \frac12]})^\ell$.
 For $\beta > 1$, we have 
 $\int_{b_\ell}^1 | h'_{\alpha, \beta}(x) | \, dx 
 \lesssim \ell^{1 + \frac{1}{\alpha}(1 - 
 \frac{1}{\beta})}$, which produces a 
 large term if $\alpha$ is small, and 
 because of this we 
 are unable to recover \eqref{eq:issue_estimate} 
 by a direct application of the argument from 
 \cite{baladi2016linear}, but instead we have  
 to develop some additional technical machinery.
\end{itemize}

Regarding (1), the method of \cite{cui2021invariant, 
liverani1999probabilistic} could be used 
instead to control $\Vert \cL_{\gamma}^n( \phi - \psi ) 
\Vert_{L^1(m)}$, but this is sufficient for 
the linear response 
formula only for parameters $(\alpha,\beta) \in \mathfrak{D}$ with 
$\alpha\beta < \tfrac{1}{1 + \beta}$
as is seen from \cite[Theorem 4.4]{
  cui2021invariant}. For parameters $\alpha\beta < 2$, 
  another option would 
  be to use the method from \cite{coates2022doubly} 
instead of \cite{korepanov2021loss}, but 
this would require a careful analysis 
of the dependence of the constants 
with respect to $\gamma$
in the upper bounds established in 
\cite{coates2022doubly}. Yet another 
plausible approach in the case $\alpha\beta < 2$ would proceed 
by an implementation of the explicit 
coupling method from 
\cite{korepanov2019explicit}. For 
$\alpha\beta < 1$, we believe that 
the approach from \cite{sarig2002subexponential,
gouezel2004sharp}
based on 
operator renewal theory could also be 
succesfully adapted.

We submit (1) and (2) above to have 
some merit if one is interested in 
deriving response formulas for non-autonomous 
or random dynamical systems. While we believe 
that the approach in this paper could be 
extended to obtain annealed or quenched 
linear response formulas associated with  
random compositions of maps 
drawn from the family \eqref{eq:cui_map}, 
we leave such extensions for future exploration.

Finally, we mention that, while we only consider 
the concrete example \eqref{eq:cui_map}, it 
can be seen from the proof that some 
generalizations are possible. For example, 
the method would apply in the setting of  
interval maps having finitely many 
full branches of the form described in  
\cite[p. 405]{cui2021invariant}.

\subsection*{Organization of the paper}
In Section \ref{sec:memory_loss} we establish 
results on the rate of memory loss and 
other technical estimates that are 
needed in the subsequent 
sections. In Section \ref{sec:cone} we 
discuss properties of the convex cones 
introduced in \cite{baladi2016linear,
cui2021invariant} and show that the 
partial derivatives given in 
Theorem \ref{thm:main} are well-defined. 
In Section \ref{sec:regularity} we study the parameter 
dependence of the transfer operator $\cL_\gamma$.
Section \ref{sec:proof_main} 
is devoted to the proof of the main 
result, Theorem 
\ref{thm:main}, for observables 
$\phi \in L^\infty$. There are two appendices: 
Appendix \ref{appendix:a} 
contains the proof of Lemma \ref{lem:cone_in_c3} 
about the invariance 
of a cone involving higher order derivatives.
In Appendix \ref{appendix:b} we explain how 
the the proof in Section \ref{sec:proof_main}
can be modified to obtain the main result 
in the case of observables $\phi \in L^q$.

%%%%%%%%%%%%%%%%%%%%%%%%%%%%%%%%%%%%
%%%%%%%%%%%%%%%%%%%%%%%%%%%%%%%%%%%%
\section{Loss of memory and distortion bounds}\label{sec:memory_loss}

Recall that  
$$
\mathfrak{D} = \{ 
(\alpha, \beta) \in 
(0, 1) \times [1, \infty)
\: : \: 
\alpha \beta < 1  
\}.
$$
Throughout this section we consider a fixed 
subset 
$$
B = [\alpha_\ell, \alpha_u] \times 
[1, \beta_u] \subset \mathfrak{D}.
$$
On $\mathfrak{D}$
we define the natural partial order 
$$
(\alpha_1, \beta_2) \le (\alpha_2, \beta_2) 
\iff \alpha_1 \le \alpha_2 \quad \text{and}
\quad 
\beta_1 \le \beta_2.
$$
Following \cite{cui2021invariant}, define 
the convex cone 
\beqn
\begin{split}
\cC_a(\gamma) 
= \{f\in L^1([0, 1]) \,:\, & \text{$f\ge 0$, $f$ decreasing,} 
\\
& \text{$x^{\alpha+1}f$ increasing, 
$ \int_0^x f(t) \, dt 
\le a x^{\frac{1}{\beta} - \alpha } m(f)$}\},
\end{split}
\eeqn
whenever $\gamma = (\alpha, \beta) \in \mathfrak{D}$. 
Denote 
by $h_\gamma$ the density of the invariant 
measure $\mu_\gamma$. It follows 
from \cite{cui2021invariant} that if 
$$
a \ge a_0(\gamma) := 
\frac{2^{\beta + 1 } ( 1 + 2^\alpha)^{ 1 + \alpha - \frac{1}{\beta} } }{
    \frac{1}{\beta} - \alpha
}, 
$$
then 
there exists $b > 0$ determined 
by $a$ and $\gamma$ such that 
\begin{align}\label{eq:cone_lb}
\cL_\gamma \cC_a(\gamma) \subset \cC_a(\gamma)  \quad 
\text{and} \quad 
\varphi \ge b  \quad \forall \varphi \in \cC_a(\gamma).
\end{align}
In particular, $h_\gamma \in \cC_a(\gamma)$ (see
\cite[p. 409]{cui2021invariant}).
Note that the cone $\cC_a(\gamma)$ is increasing in $a$, and that 
$$
\gamma_1 \le \gamma_2 \implies 
a_0(\gamma_1) \le a_0(\gamma_2) \quad \text{and} \quad 
\cC_a(\gamma_1) \subset \cC_a(\gamma_2).
$$

We can slightly generalize the invariance property of 
$\cC_a(\gamma)$ stated above:

\begin{lem}\label{lem:inv_of_lsv_cone} For all 
    $\gamma = (\alpha,\beta) 
    \in \mathfrak{D}$ with $\gamma \le 
\gamma_u = (\alpha_u, \beta_u)$, 
$$
\cL_\gamma \cC_a(\gamma_u) \subset 
\cC_a(\gamma_u)
$$
holds if $a \ge a_0(\gamma_u)$.
%$$
%a \ge a_0(\gamma_u) := \frac{2^{3+\alpha_u - 
%\frac{1}{\beta_u} }}{
%  \frac{1}{\beta_u} - \alpha_u
%}.
%$$
%$$
%a \ge a_0(\gamma_u) := 2( \tfrac{1}{\beta_u} 
%- \alpha_u )^{-1}.
%$$
%Moreover, 
%$$
%\inf_{x\in[0,1]} |\varphi(x)|
%\ge \min \biggl\{ a, 
%  \biggl[ \tfrac{1}{ a^{\alpha} (  a^{1- \frac{1}{\beta} } 
%  + \frac{1}{\alpha} ) } \biggr]^{ 
%    \frac{1}{\frac{1}{\beta} - \alpha}
%   }
% \biggr\} m(\varphi), \qquad \forall \varphi \in C_a(
%    \gamma). 
%$$
\end{lem}

\begin{proof} Let 
  $\varphi \in \cC_a(\gamma_u)$. 
  It is easy to see that  
  $\cL_\gamma \varphi \ge 0$,  
  $\cL_\gamma \varphi$ is decreasing, 
  and 
  $x^{\alpha_u + 1} \cL_\gamma \varphi(x) $ 
  is increasing. Next, note that 
  $T_{\gamma}^{-1}[0,x] \subset 
  T_{\gamma_u}^{-1}[0,x]$ 
  holds for all $x \in [0,1]$. Hence, 
  $$
  \int_0^x \cL_\gamma \varphi(t) \, dt
  = \int_{T^{-1}_\gamma[0,x]} \, \varphi(t) \, dt
  \le \int_{ T_{ \gamma_{{}_u} }^{-1} [0,x]  } \, 
  \varphi(t) \, dt
  = \int_{0}^x \, 
  \cL_{\gamma_u} \varphi(t) \, dt 
  \le a m(\varphi) x^{ \frac{1}{\beta_u} - \alpha_u },
  $$
  provided that 
  $a \ge a_0(\gamma_u)$.
  %The latter is 
  %esily obtained using the argument in [LSV, Lemma 
  %2.4]
  %\note{The proof in LSV seems to require 
  %$\beta \ge 1$}
\end{proof}

For all $n \ge 0$, define 
$$
b_n = f_{\gamma, 1}^{-n}( \tfrac12 ).
$$
Moreover, let 
$\hat{b}_0 = 1$ and for $n \ge 1$ let $\hat{b}_n 
\in [\tfrac12, 1]$ be such that 
$$
T_\gamma ( \hat{b}_n ) = f_{\gamma, 1}(b_n).
$$
Then, for every $n \ge 0$, 
$T_\gamma^{n+1}$ maps both $(b_{n+1}, b_n)$ 
and $(\hat{b}_{n+1}, \hat{b}_n)$ bijectively 
onto $(\tfrac12,1)$, and 
\begin{align}\label{eq:b_n_rel}
\hat{b}_{n+1} - \tfrac12 = \tfrac{1}{2} b_{n}^{
  \frac{1}{\beta}
}.
\end{align}

%By [Korepanov, Lemma 5.2] and \eqref{eq:b_n_rel} 
%we have the following 
%estimates on $b_n$ and $\hat{b}_n$:

The following bounds on $b_n$ and $\hat{b}_n$
are immediate consequences of Lemma 5.2 
in \cite{korepanov2016linear}
combined with \eqref{eq:b_n_rel}:

\begin{lem}\label{lem:bn_bound} For all $n \ge 1$,
    \begin{align*}
    \biggl[ \frac{
      1
    }{2^\alpha + n \alpha
    2^{ \alpha - 1 }
    } \biggr]^{ \frac{1}{\alpha} }
    \le 
    b_n \le \biggl[ \frac{
      1
    }{2^\alpha + n \alpha (1 - \alpha )
    2^{ \alpha - 1 }
    } \biggr]^{ \frac{1}{\alpha} }
    \end{align*}
    and
    \begin{align*}
    \tfrac12 \biggl[ \frac{
      1
    }{2^\alpha + n \alpha
    2^{ \alpha - 1 }
    } \biggr]^{  \frac{1}{\alpha\beta} }
    \le 
    \hat{b}_{n+1} - \tfrac12 \le 
    \tfrac12 \biggl[ \frac{
      1
    }{2^\alpha + n \alpha (1 - \alpha )
    2^{ \alpha - 1 }
    } \biggr]^{ \frac{1}{\alpha\beta}}.
    \end{align*}
    Moreover, there exist $C_i = C_i(\alpha, \beta) 
    > 0$ such that 
    \begin{align*}
      C_1 n^{- \frac{1}{\alpha} - 1}  \le 
       b_{n} - b_{n+1} \le C_2 n^{- \frac{1}{\alpha} 
    - 1 } \quad \text{and} \quad 
    C_3 
    n^{ - \frac{1}{\alpha\beta} - 1   } \le 
    \hat{b}_{n} - \hat{b}_{n+1}  \le C_4 
    n^{ - \frac{1}{\alpha\beta} - 1   }.
    \end{align*}
\end{lem}

%\begin{proof} The inequalities in  
%  \eqref{eq:korepanov_1} are from [Korepanov, 
%  Lemma 5.2]. Both inequalities in 
%  \eqref{eq:korepanov_2} follow immediately 
%  from \eqref{eq:korepanov_1} combined 
%  with \eqref{eq:b_n_rel}. The estimates 
%  in \eqref{eq:korepanov_corollary} 
%  follow from \eqref{eq:korepanov_1} and 
%  \eqref{eq:korepanov_2} together with 
%\end{proof}

\begin{lem}\label{lem:distortion} 
    There exist $C_i > 0$ such that, 
    for all 
    $\gamma \in \mathfrak{D}$ and all $n \ge 1$,
    \begin{align}\label{eq:distortion_1}
    | \log (T^n_\gamma)'y - 
    \log (T^n_\gamma)'x |
    \le C_1  |  T^n_\gamma(y) 
    - T^n_\gamma(x) |  \quad 
    \forall x,y \in (b_{n}, b_{n-1})
    \end{align}
    and 
    \begin{align}\label{eq:distortion_2}
        | \log (T^n_\gamma)'y - 
        \log (T^n_\gamma)'x |
        \le C_2  |  T^n_\gamma(y) 
        - T^n_\gamma(x) | \quad 
    \forall x,y \in (\hat{b}_{n}, \hat{b}_{n-1}).
    \end{align}

\end{lem}

\begin{proof} Since both branches of 
$T_{\gamma}$ have negative Schwarzian derivative, 
both upper bounds follow 
easily by applying the 
Koebe Principle \cite[Theorem IV.1.2]{melo2012one}: see 
 \cite[Lemma 4.8]{bahsoun2014decay},  
  \cite[Section 2]{aimino2015polynomial}, or
\cite[p. 416]{cui2021invariant}.    
\end{proof}

\subsection{Loss of memory}\label{sec:loss} In this section we will show show 
a key lemma by applying a result from 
\cite{korepanov2021loss}. We start by recalling 
some definitions from \cite{korepanov2021loss}.

Let $Y = [\tfrac12, 1]$ and $\gamma \in B$. 
%For brevity, denote $T = T_\gamma$.
We consider the 
measurable 
partition 
$\cP = \cP(\gamma)$ of $X = [0,1]$ consisting of all the intervals 
$I_n = (b_{n+1}, b_{n})$ and 
$J_n = (\hat{b}_{n+1}, \hat{b}_{n})$ for 
$n \ge 0$. Then the first return time 
$$
\tau(x) = \inf \{ n \ge 1 \: : \: T^n_\gamma(x)
\in Y \}
$$
is constant on each $a \in \cP$: 
$\tau(a) := \tau |_a(x) = n + 1$ for 
$x \in a \in \{ I_n, J_n \}$. Define 
the first return map  
$$
F_a = T^{\tau(a)}_\gamma : a \to Y, 
$$
and denote by $m_1$ the Lebesgue measure 
on $Y$ normalized to probability. 
Note that $F_a$ sends $a$ diffeomorphically onto 
$Y$.
For a 
non-negative function $\rho : Y \to \bR_+$, 
denote by $| \rho |_{\text{LL}}$ 
the Lipschitz seminorm 
of the logarithm of $\rho$:
$$
|\rho|_{\text{LL}}
= \sup_{y \neq y' \in Y}
\frac{| \log \rho(y) -  \log \rho(y') |}{|y - y'|},
$$
where we use the conventions $\log 0 = -\infty$ 
 and $\log 0 - \log 0 = 0$. 
 %Typically 
 %$\rho$ will be the Radon-Nikodym derivative 
 %of a non-negative measure supported on $Y$.

 \begin{lem}\label{lem:ml_conditions} There 
    exist $C, K,K', \delta > 0$ and 
$\lambda > 1$ depending only on $B$, such that 
the following hold for all $\gamma \in B$:
\begin{itemize}
    \item[(i)] 
    For all 
    $y, y' \in a$ and $a \in 
    \{ J_n \}_{n=0}^\infty$,
    $| F_a(y) - F_a(y') | \ge 
    \lambda | y - y' |$.
    %Moreover, the density 
    %$$
    %\zeta = \frac{ d (F_a)_* (m_1 |_a) }{dm_1} 
    %$$
    %satisfies $|\zeta|_{LL} \le K$. 
    \smallskip
    \item[(ii)] The map $F_a$ is non-singular 
    with Log-Lipschitz Jacobian:
    $$
    \zeta_a = \frac{d(F_a)_*(m_1|_a)}{dm_1} 
    \quad \text{satisfies} \quad |\zeta_a|_{
        \textnormal{LL}
    } \le K.
    $$
    \item[(iii)] For all $x, x' \in a \in \cP$,
    $$
    \max_{0\le j \le \tau(a)}
    | T^j_\gamma(x) - T^j_\gamma( x') |
    \le K' | F_a(x) - F_a ( x' )  |.
    $$
    \smallskip
    \item[(iv)] For all $n \ge 0$, $m_1(T^{-n}_\gamma Y) \ge \delta$. \smallskip
    \item[(v)] For all $n \ge 1$, 
    $$
    m_1(\tau \ge n) \le C 
    n^{ - \frac{1}{\alpha_u \beta_u} }.
    $$
\end{itemize}
\end{lem}

\begin{proof} For brevity, denote 
  $T = T_\gamma$. 

  \noindent \textbf{(i)} 
  We need to show that there exists 
  $\lambda > 1$ such that 
  for all $n \ge 0$,
  \begin{align}\label{eq:ind_lb}
  (T^{n+1})'(x) > \lambda \quad \forall x \in J_n.
  \end{align}
  The inductive proof on p. 417 in 
  \cite{cui2021invariant} shows 
  that \eqref{eq:ind_lb} holds with $\lambda = \tfrac32$.
  For completeness, let us recall the computation:
  for 
  $n=0$, using piecewise 
  convexity of $T$ together with \eqref{eq:b_n_rel},
  $$
  T'(x) \ge T'(\hat{b}_1) 
  \ge \beta 2^{\frac{1}{\beta}} > \tfrac32 
  \quad \forall \beta > 0.
  $$
  For $n=k+1$, assuming \eqref{eq:ind_lb}
  for $n = k$,
  $$
  (T^{k+2})'(x) 
  \ge (T^{k+2})'( \hat{b}_{k+2} )
  \ge \frac{T'(\hat{b}_{n+2})
  T'(b_{n+1})}{ T'(\hat{b}_{n+1}) } \lambda  
  = \biggl( \frac{b_{n+1}}{b_n} \biggr)^{
    \frac{\beta - 1 }{\beta}
  } 
  (  1 + 2^{\alpha} (1 + \alpha)
  b_{n+1}^\alpha ) \lambda,
  $$
  where the inductive hypothesis and 
  \eqref{eq:b_n_rel} were used. Hence, 
  for $t = 2^{\alpha}b_{n+1}^\alpha$, 
  $$
  (T^{k+2})'(x) \ge \biggl(\frac{1}{1+t} 
  \biggr)^{ \frac{\beta - 1 }{\beta}  }
  ( 1 + (1+\alpha) t ) \lambda 
  > \frac{1 + (1 + \alpha)t}{1 + t} \lambda > \lambda.
  $$

\noindent\textbf{(ii)} For $a = J_n$ 
        and $z \in Y$,   
        $$
        \zeta_a(z) = 
        \frac{1}{
          (T^{n+1})' ( z_n ),
        }
        $$
        where $z_n = ( T^{n+1} |_{J_n}  )^{-1} z$. By 
        \eqref{eq:distortion_2}, 
        $$
        | \log (T^{n+1})' ( z_n ) 
        - \log (T^{n+1})' ( z_n' ) |
        \le C |  z - z' | \quad 
        \forall z,z' \in Y. 
        $$

\noindent \textbf{(iii)} Using (i), it is easy to 
see that this holds with 
$K' = 1$. \smallskip 

\noindent \textbf{(iv)} Recalling \eqref{eq:cone_lb}, this follows from $\cL^n( \mathbf{1} ) \ge b > 0$ 
for all $n \ge 0$. The lower bound 
$b$ depends only on $B$ since 
$\mathbf{1} \in \cC_a(\gamma_u)$ for $a \ge 1$.

\noindent \textbf{(v)} By Lemma \ref{lem:bn_bound}, 
for $n \ge 2$, 
\begin{align*}
  m_1( \tau \ge n ) 
  = 2 m ( [  \tfrac12, \hat{b}_{n-1}  ]   )
  \le 2 (
  \hat{b}_n - \tfrac12 ) \le 
        2^{ \frac{1}{\beta}( \frac{1}{\alpha^2} + 
        \frac{2}{\alpha} ) }   
    n^{ - \frac{1}{\alpha \beta} }
    \le 
    C(B)  
    n^{ - \frac{1}{\alpha_u \beta_u} }.
\end{align*}
\end{proof}

%\note{Ultimately should 
%also try to extend this 
%for parameters $\beta < 1$}

Now, if $\gamma \in B$ and $\mu$ is a non-negative measure supported 
on 
$Y$, it follows that for $\lambda > 1$ and $K$ 
as in the previous lemma,
$$
\biggl|  \frac{d (F_a)_* (\mu |_a)}{dm_1} 
\biggr|_{\text{LL}} 
\le K + \lambda^{-1} | \mu |_{\text{LL}}.
$$
For a proof of this fact, see 
\cite[Proposition 3.1]{korepanov2019explicit}.
An immediate consequence of this 
is the following 
result, where we may choose 
any $K_2 > (1 - \lambda^{-1})^{-1}$
and $K_1 = K + \lambda^{-1}K_2$:

\begin{lem}\label{lem:k1_k2} There exist constants $0 < K_1 < K_2$ 
depending only on $B$ such that for each 
$\gamma \in B$, for each non-negative measure 
$\mu$ on $Y$ with $|\mu|_{\textnormal{LL}} \le K_2$, 
and each $a \in \cP(\gamma)$ with $a \subset Y$, 
$$
\biggl|  \frac{d (F_a)_* (\mu |_a)}{dm_1} 
\biggr|_{\textnormal{LL}}  \le K_1.
$$
The constants $K_1$ and $K_2$ can be chosen 
arbitrarily large.
\end{lem}

Fix  
$K_1$ and $K_2$ as in Lemma 
\ref{lem:k1_k2}. The following definition 
is an adaptation of 
\cite[Definition 3.5]{korepanov2021loss}.

\begin{defn} We say that a non-negative measure 
$\mu$ on $X$ is regular, if for every 
$\gamma \in B$ and every $a \in \cP(\gamma)$, 
$$
\biggl|  \frac{d (F_a)_* (\mu |_a)}{dm_1} 
\biggr|_{\textnormal{LL}}  \le K_1.
$$
%We say that $\mu$ has tail bound 
%$r$, with $r : \{0,1,\ldots \} \to [0, \infty)$, 
%if for all $n \ge 0$, 
%$$
%\mu ( \{ x \in X \: : \: 
%    T^k_\gamma(x) \notin Y \text{ for all 
%    $1 \le k < n$} \}  ) \le r(n).
%$$
%tail bound definition does not seem to work here
\end{defn}

\begin{remark}
    Note that the measure 
    $(F_a)_* (\mu |_a)$ is supported on 
    $Y$.
\end{remark}

\begin{lem}\label{lem:regular} For sufficiently large $K_1$ 
    in Lemma \ref{lem:k1_k2}, the following 
    holds: If $\gamma_0 = (\alpha_0, \beta_0) 
    \in B$ and $\mu$ is a
    non-negative 
    measure on $X$ whose density  belongs to 
    $C_a(\gamma_0)$ with $a \ge a_0(\gamma_u)$, 
    then $\mu$ is regular. Moreover, there 
    exists $C = C(a, B) > 0$ such that for all 
    $\gamma = (\alpha, \beta)\in B$, 
    \begin{align}\label{eq:mu_tail}
\mu ( \{ x \in X \: : \: 
    T^k_{\gamma}(x) \notin Y \text{ for all 
    $1 \le k < n$} \}  ) \le 
    C n^{-\gamma_*},
    \end{align}
    where 
    $$
    \gamma_* = \tfrac{1}{\alpha_u}(
        \tfrac{1}{\beta_u} - \alpha_0 ).
    $$
\end{lem}

\begin{proof} 
        To see that $\mu$ is 
        regular, let $\rho \in C_a(\gamma_0)$ 
        denote the 
        density of $\mu$. Let $\gamma \in B$. 
        We consider the 
        case $a \in \{J_n\}_{n=0}^\infty 
        \subset \cP(\gamma)$, 
        the proof for $a \in \{I_n\}_{n=0}^\infty$ being 
        similar, which we omit. For $a = J_n$ 
        and $z \in Y$,   
        $$
        \frac{d( F_a )_* ( \mu |_a )}{dm_1} (z) = 
        \frac{\rho(  z_n )}{
          (T^{n+1}_\gamma)' ( z_n ),
        }
        $$
        where $z_n = ( T^{n+1}_\gamma |_{J_n}  )^{-1} z$. By 
        \eqref{eq:distortion_2}, 
        $$
        | \log (T^{n+1}_\gamma)' ( z_n ) 
        - \log (T^{n+1}_\gamma)' ( z_n' ) |
        \le C |  z - z' | \quad 
        \forall z,z' \in Y. 
        $$
        It remains to show that 
        \begin{align*}
          | \log \rho(z_n) - 
          \log \rho(z_n') |
          \le C | z - z'|.
        \end{align*}
        Suppose that $z > z'$. Then
        $z_n > z_n'$. Since $\rho$ is 
        decreasing, 
        $$
        | \log \rho(z_n) - 
          \log \rho(z_n') |
        = \log \frac{\rho(z_n')}{\rho(z_n)} 
        $$
        and, since $x^{\alpha_0+1} \rho(x)$ is 
        increasing, 
        $$
        \frac{\rho(z_n')}{\rho(z_n)} 
        = \frac{ (z_n')^{\alpha_0 + 1} \rho(z_n') }{
          (z_n)^{\alpha_0+ 1} \rho(z_n)
        } \frac{ (z_n)^{\alpha_0+ 1} }{(z_n'
        )^{\alpha_0 + 1}}
        \le 
        \frac{ (z_n)^{\alpha_0 + 1} }{(z_n')^{
            \alpha_0 + 1}}.
        $$
        %From the distortion bound \eqref{eq:distortion_2} 
        %we obtain 
        %$$
        %z_n - z_n' \le C |J_n|^{-1} (z - z').
        %$$
        Hence, 
        $$
        | \log \rho(z_n) - 
          \log \rho(z_n') |
          \le 2 (  \log z_n - \log z_n' )
          \le 2 \tfrac{1}{z_n'} | z_n - z_n' |
          \le 4 | z_n - z_n' |
          \le C | z - z'|,
        $$
        where the last inequality follows 
        again by 
        \eqref{eq:distortion_2}. We have shown that 
        $\tfrac{d( F_a )_* ( \mu |_a )}{dm_1} (z)
        \le C$ for some $C = C(B) > 0$. The regularity 
        of $\mu$ follows by taking $K_1$ sufficiently 
        large in Lemma \ref{lem:k1_k2}.
        
        It remains to show the tail bound  
        \eqref{eq:mu_tail}.
        Since $\rho \in \cC_a(\gamma_0)$, 
        \begin{align*}
          &\mu ( \{ x \in X \: : \: 
            T^k_{\gamma}(x) 
            \notin Y \text{ for all 
            $1 \le k < n$} \}  )
            = \mu( [ 0, b_{n-1} ] ) + 
            \mu( [ \tfrac12, \hat{b}_{n-1} ] )
         \\ 
         &= \int_0^{b_{n-1}} \rho(x) \, dx 
          + \int_{\frac12}^{\hat{b}_{n-1}} 
           \rho(x) \, dx 
        \le a b_{n-1}^{  \frac{1}{\beta_0} - \alpha_0 }
        +   \rho(\tfrac12) (  \hat{b}_{n-1} - \tfrac12 ) \\
        &\le a C(B)  ( b_{n-1}^{  \frac{1}{\beta_0} 
        - \alpha_0}
        +   ( \hat{b}_{n-1} - \tfrac12 ) ).
        \end{align*}
        Now, by applying 
        Lemma \ref{lem:bn_bound} we obtain 
        \begin{align*}
          &\mu ( \{ x \in X \: : \: 
            T^k_\gamma (x) \notin Y \text{ for all 
            $1 \le k < n$} \}  ) \\
            &\le 
            a C(B) (  
                n^{  - \frac{1}{\alpha}(
                \frac{1}{\beta_0} - \alpha_0 ) } 
                + n^{ -\frac{1}{\alpha\beta} } )
                \le 
                a C(B) (  
                    n^{  - \frac{1}{\alpha_u}(
                    \frac{1}{\beta_u} - \alpha_0 ) } 
                    + n^{ -\frac{1}{\alpha_u\beta_u} 
                    } ).
        \end{align*}
\end{proof}

We are now ready to establish the key lemma 
mentioned at the beginning of this section:

\begin{lem}[Loss of memory]\label{lem:memory_loss}
    For all $\gamma_0 = (\alpha_0, \beta_0) 
    \in B$ there 
    exists a constant $C > 0$ determined by 
    $a, B, \gamma_0$ such that the 
    following holds: If  
$\phi, \psi \in 
  \cC_a(\gamma_0)$ are 
  such that $m(\phi) = m(\psi)$,
  where $a \ge a_0(\gamma_u)$, then for 
  all $\gamma \in B$, 
  \begin{align}\label{eq:ml_strong}
    \Vert \mathcal{L}_{\gamma}^n 
    ( \phi - \psi   ) \Vert_{L^1(m)}
    \le 
    C ( m(\psi) + 
    m(\phi) ) n^{-\gamma_*},
  \end{align} 
  where 
  $$
    \gamma_* = \tfrac{1}{\alpha_u}(
        \tfrac{1}{\beta_u} - \alpha_0 ) 
    $$
  Moreover, there exists a constant 
  $C' > 0$ determined by $a, B$ such that 
  for all $\gamma \in B$, and all 
  $\phi, \psi \in 
  \cC_a(\gamma_u)$ with  $m(\phi) = m(\psi)$,
  where $a \ge a_0(\gamma_u)$, 
  \begin{align}\label{eq:ml_weak}
    \Vert \mathcal{L}_{\gamma}^n 
    ( \phi - \psi   ) \Vert_{L^1(m)}
    \le 
    C' ( m(\psi) + 
    m(\phi) ) n^{ 1 - \frac{1}{\alpha_u\beta_u} }.
  \end{align}
\end{lem}

\begin{remark}
    If $\alpha_0 \ll \alpha_u$, \eqref{eq:ml_strong} 
    gives a much better estimate than 
    \eqref{eq:ml_weak}.  
    This
    will be instrumental in the proof of 
    the linear response formula for 
    $\alpha\beta \ge 2$. 
\end{remark}

%\begin{remark}
%    For $\gamma_1 = \gamma_2 = \gamma = (\alpha, \beta)$, 
%    the first upper bound is an improvement of the 
%    rate $O(n^{ \frac{1}{\beta} 
%    ( 1- \frac{1}{\alpha 
%    \beta}  ) })$ established in \cite{cui2021invariant}.
%    It follows that the central limit theorem given in 
%    \cite[Corollary 4.5]{cui2021invariant} extends 
%    from $\alpha\beta < \tfrac{1}{1+\beta}$
%    to the regime 
%    $\alpha\beta < 2$, 
%    since \eqref{eq:ml_1} implies that correlations associated to 
%    H\"{o}lder continuous observables are 
%    summable for if $\alpha\beta < 2$.
%\end{remark}

\begin{proof}[Proof of Lemma 
    \ref{lem:memory_loss}] 
    Both upper bounds follow by applying  
    \cite[Theorem 3.8]{korepanov2021loss} in 
    the case of the family $\cT =  \{ T_\gamma \}$ 
    consisting only of the single map $T_\gamma$ with $\gamma \in B$.
    We provide more details below.

    By Lemma \ref{lem:ml_conditions}, the 
    family $\cT$ satisfies the conditions 
    (NU:1)-(NU:7) 
    from \cite[Section 3.1]{korepanov2021loss} 
    with constants depending only on $B$.  
    Then, suppose that $\mu$ and $\mu'$ are two 
    probability measures on $X$ with densities 
    $\rho, \rho' \in \cC_a(\gamma_0)$, where 
    $a \ge a_0(\gamma_u)$. By Lemma 
    \ref{lem:regular}, $\mu$ and $\mu'$ are 
    regular. By Lemmas \ref{lem:ml_conditions} 
    and \ref{lem:regular}, for some 
    $C_i = C_i(B)$ > 0, 
    $$
    m_1( \tau \ge n) \le C_1 n^{-\frac{1}{\alpha_u 
    \beta_u}},
    $$
    and for $\nu \in \{ \mu, \mu' \}$,
    \begin{align}\label{eq:tail_nu}
    \nu ( \{ x \in X \: : \: 
            T^k_\gamma (x) \notin Y \text{ for all 
            $1 \le k < n$} \}  ) \le 
            C_2n^{-\gamma_*}.
    \end{align}
    In the terminology of \cite{korepanov2021loss}, 
    this means that $m_1$ has tail bound 
    $h(n) = C_1(B) n^{-\frac{1}{\alpha_u 
    \beta_u}}$ and $\nu$ has 
    tail bound $r(n) = C_2(B) n^{-\gamma_*}$, 
    where $\gamma_* \le \tfrac{1}{\alpha_u 
    \beta_u}$. We have verified the 
    assumptions of 
    \cite[Theorem 3.8]{korepanov2021loss}, 
    which implies the existence of 
    a constant $C = C(\gamma_*, B) > 0$ 
    such that 
$$
\sup_{A \in \cB([0,1]) } 
\biggl| 
  \int_A \cL^n_\gamma ( \rho - \rho'  ) \, dm 
\biggr| = 
\sup_{A \in \cB([0,1]) } | (T^n_\gamma)_* \mu(A) -  
(T^n_\gamma)_* \nu (A)  | 
\le C n^{   - \gamma_*  }
$$
holds for all $n \ge 1$ and all 
$\gamma \in B$. In particular, 
\begin{align}\label{eq:density_ml}
\Vert \cL^n_\gamma ( \rho - \rho'  ) \Vert_{L^1(m)}
\le 2C n^{ -  \gamma_* }.
\end{align}
Now, \eqref{eq:ml_strong} follows
from \eqref{eq:density_ml} by scaling $\phi$ and $\psi$
with $m(\phi)$, noting that 
$\phi / m(\phi) \in \cC_a(\gamma_u)$.

Finally, \eqref{eq:ml_weak} also follows by 
an application of Theorem 
\cite[Theorem 3.8]{korepanov2021loss}, because 
we can replace the right hand side of 
\eqref{eq:tail_nu}  with  
$C_2n^{1 - \frac{1}{\alpha_u\beta_u}}$ for 
densities $\rho, \rho' \in \cC_a(\gamma_u)$.
In this case, \cite[Theorem 3.8]{korepanov2021loss} 
guarantees that $\Vert \cL^n_\gamma ( \rho - \rho'  ) \Vert_{L^1(m)}
\le C' n^{1 - \frac{1}{\alpha_u\beta_u}}$ 
holds for $C' > 0$ determined by $B$ 
and $a$.

\end{proof}

%The following distortion bounds 
%will be needed in the proof of 
%Theorem \ref{thm:main}:

\subsection{Technical estimates} We assume 
that $\gamma = (\alpha,\beta) \in \mathfrak{D}$ 
and denote $T = T_\gamma$.

\begin{lem}\label{lem:distortion_returns}  
    For all $\ell \ge 1$, 
  there exist $C_i = C_i(\gamma, \ell) > 0$
  such that for all $m \ge 1$ the following bounds hold:
\begin{align}
  &\sup_{ x \in (0,1) : T^m( x ) \ge b_\ell } \frac{1}{(T^m)' x} 
  \le C_1
  m^{  - 1 - \frac{1}{\alpha\beta} }, \label{eq:dist_first} \\
  &\sup_{ x \in (0,1) : T^m
  ( x ) \ge b_\ell } 
  \frac{  (T^m)'' x  }{ ((T^m)' x)^2 } \le  
  C_2. \label{eq:dist_second}
\end{align}
\end{lem}

\begin{proof} 
  Suppose that $m, \ell \ge 1$ and $x \in [0, 1]$ are such that $T^m x \ge b_\ell$. Let 
  $0 \le \ell_1 < \ldots < \ell_k \le m$ be such that for all $0 \le j \le m$, 
  $T^jx \in Y = [\tfrac12, 1]$ if and only if 
  $j \in \{ \ell_1, \ldots, \ell_k\}$. 
  We may assume $k \ge 1$, for otherwise 
  the trajectory remains in $[0, \tfrac12]$ and we 
  are in the situation of the usual 
  intermittent 
  map $T_{\alpha, \beta}$ with $\beta = 1$,
   in which case both upper bounds 
  (are well-known and) follow easily from  
  \eqref{eq:distortion_1}.
  
  We have 
  $$
  (T^m)'x = (T^{m - \ell_k})' T^{\ell_k} x \cdot \prod_{j=2}^k ( T^{\ell_j - \ell_{j-1}} )' T^{\ell_{j-1}} x
  \cdot (T^{\ell_1})' x, 
  $$
  where we use the convention $T^0 = \text{id}_{[0,1]}$. Suppose $\ell_1 > 0$ 
  i.e. $x \in [0, \tfrac12]$. Since $T^j x \in [0, \tfrac12]$ for $j < \ell_1$
  and $T^{\ell_1} x \ge \tfrac12$, we have the lower bound 
  %$(T^{\ell_1})' x \ge 1$ and 
  \begin{align}\label{eq:return_lb_1}
  (T^{\ell_1})' x \ge (f^{\ell_1}_{\gamma, 1} )'
  b_{\ell_1} \ge C(\alpha) \ell_1^{ 
    1 + \frac{1}{\alpha} }, 
  \end{align}
  where the last inequality follows from 
  Lemma 5.3 in 
  \cite{korepanov2016linear}. Next, note that 
  for each $2 \le j \le k$ 
  we must have $T^{\ell_{j-1}}x \in [ \hat{b}_{p_j}, \hat{b}_{p_j - 1} ]$, where $p_j = \ell_j - \ell_{j-1}$. 
  Hence, $( T^{\ell_j - \ell_{j-1}} )' T^{\ell_{j-1}} x > \lambda > 1$ holds 
  by \eqref{eq:ind_lb}. On the other hand, it follows by an application of Lemma \ref{lem:distortion} that 
  \begin{align}\label{eq:lb_deriv_kobe}
  ( T^{\ell_j - \ell_{j-1}} )' T^{\ell_{j-1}} x \ge C 
   \frac{T^{p_j} ( \hat{b}_{p_j - 1}) - T^{p_j} ( \hat{b}_{p_j })}{
     \hat{b}_{p_j - 1} - \hat{b}_{p_j} 
    }
  = \frac{C}{2} \frac{1}{ \hat{b}_{p_j - 1} - \hat{b}_{p_j}   }.
  \end{align}
  Using \eqref{eq:b_n_rel}, the 
  H\"{o}lder continuity of 
  $t \mapsto t^{\frac{1}{\beta}}$, 
  and Lemma \ref{lem:bn_bound} we see that $  \hat{b}_{p_j - 1} - \hat{b}_{p_j}  \le 
     C(\gamma) ( \ell_{j} - \ell_{j-1} )^{ -  1 - \frac{1}{\alpha\beta}  } $, 
     so that 
  $$
  ( T^{\ell_j - \ell_{j-1}} )' T^{\ell_{j-1}} x  \ge C(\gamma) 
  ( \ell_{j} - \ell_{j-1} )^{   1 + \frac{1}{\alpha\beta}  }.
  $$
  Hence, 
  \begin{align}\label{eq:return_lb_2}
    ( T^{\ell_j - \ell_{j-1}} )' T^{\ell_{j-1}} x  \ge 
    \max \{ \lambda , 
    C(\gamma) 
  ( \ell_{j} - \ell_{j-1}  )^{    1 + \frac{1}{\alpha\beta}  } \}.
  \end{align}
  Finally, to control $(T^{m - \ell_k})' T^{\ell_k} x$, note that if $\ell_k < m$
  we must have 
  $ \hat{b}_{m - \ell_k + \ell } 
  \le T^{\ell_k}x < \hat{b}_{m - \ell_k} $. Hence, 
  there exists $0 \le p < \ell$ such that $T^{\ell_k} x \in J_{ m - \ell_k + p  } =
   [ \hat{b}_{m - \ell_k + p + 1  }, 
  \hat{b}_{m - \ell_k + p }  ]$. Since $T^{m- \ell_k}$ maps $J_{m-\ell_k + p}$ diffeomorphically 
  onto $[b_{p+1}, b_{p}]$ and both branches of $T$ 
  have negative Schwarzian derivative, it 
  follows by an application of 
  the Koebe Principle 
  \cite[Theorem IV.1.2]{melo2012one} that
  \begin{align}\label{eq:return_lb_3}
  ( T^{m - \ell_k} )' T^{\ell_k}x \ge C(\ell) 
  \frac{ b_{  p  } - b_{ p + 1 } }{ \hat{b}_{m - \ell_k + p   } - 
  \hat{b}_{m - \ell_k + p + 1  } }
  \ge C(\gamma, \ell) 2 b_\ell  ( m - \ell_k)^{   1 + \frac{1}{\alpha\beta} 
   }.
  \end{align}

  Gathering the estimates \eqref{eq:return_lb_1}, \eqref{eq:return_lb_2}, and 
  \eqref{eq:return_lb_3}, we see that 
  for some $C_i = C_i(\gamma, \ell) > 0$ and $\lambda > 1$,
  \begin{align}\label{eq:dist_sharp}
   (T^m)'x &\ge  
      C_1   ( m - \ell_k + 1)^{   1 + \frac{1}{\alpha\beta} 
    }  (\ell_1 + 1)^{  1 + \frac{1}{\alpha} } 
   \cdot \prod_{j=2}^k \max \{ \lambda , 
   C_2 
 ( \ell_{j} - \ell_{j-1}  )^{  1 + \frac{1}{\alpha\beta }  } \}. 
 %\\
 %& \ge C(\alpha, \beta)   b_\ell ( m - \ell_k + 1)^{  \frac{1}{\beta} ( 1 + \frac{1}{\alpha} 
 %) }  (\ell_1 + 1)^{  1 + \frac{1}{\alpha} } r^{  \frac{k}{2} }  
 %\prod_{j=2}^k
 %C_2 
 %( \ell_{j} - \ell_{j-1} + 1 )^{  \frac{1}{2\beta} (  1 + \frac{1}{\alpha} ) }.
\end{align}
One of the integers $\ell_1, \ell_2 - \ell_1, \ldots, \ell_k - \ell_{k-1}, m - \ell_k$ must be 
at least $ \tfrac{m}{k+1}$, so that 
$$
(T^m)'x \ge C(\gamma, \ell)   \biggl( \frac{m}{k+1}  \biggr)^{ 
   1 + \frac{1}{\alpha\beta}    } \lambda^{k} 
  \ge C(\gamma, \ell)  m^{  1 + \frac{1}{\alpha\beta}   }.
$$
This proves \eqref{eq:dist_first}. 

For \eqref{eq:dist_second}, we write 
$$
\frac{(T^m)''x}{  ((T^m)'x)^2 } 
= \sum_{j=0}^{m-1} \frac{T''(T^j x )}{ T'( T^j x ) } 
\cdot \frac{ 1 }{ (T^{m - j })' T^j x }
$$
and consider elements in the sum case by case.

\noindent\textbf{Case} 
$1^{\circ}$: $j < \ell_1$. Since 
$T^j x \ge b_{\ell_1 - j}$ and 
$T^j x \in [0, \tfrac12]$,
\begin{align*}
\frac{T''(T^j x )}{ T'( T^j x ) } 
\cdot \frac{ 1 }{ (T^{m - j })' T^j x } 
&\le C(\alpha)
b_{\ell_1 - j}^{\alpha - 1} 
\frac{1}{(T^{m - \ell_1} )' T^{\ell_1} x } 
\cdot \frac{1}{ (T^{\ell_1 - j})' T^j } \\
&\le 
C(\alpha)
( \ell_1 - j)^{ \frac{1}{\alpha} - 1 } 
\frac{1}{(T^{m - \ell_1} )' T^{\ell_1} x } 
\cdot \frac{1}{ (T^{\ell_1 - j})' T^jx },
\end{align*}
Using 
\eqref{eq:dist_first} and \eqref{eq:return_lb_1}, 
we see that 
$$
\frac{1}{(T^{m - \ell_1} )' T^{\ell_1} x }  
\le C(\gamma, \ell)
(m - \ell_1)^{- 1  - \frac{1}{\alpha \beta }  } 
\quad \text{and} \quad 
\frac{1}{ (T^{\ell_1 - j})' T^jx } 
\le C(\gamma) ( \ell_1 - j )^{ - 1 - \frac{1}{\alpha} }.
$$
Hence, 
$$
\frac{T''(T^j x )}{ T'( T^j x ) } 
\cdot \frac{ 1 }{ (T^{m - j })' T^j x } 
\le C(\gamma, \ell) 
(\ell_1 - j)^{-2} (m - \ell_1)^{-
   1 - \frac{1}{\alpha\beta}  }.
$$

\noindent\textbf{Case} 
$2^{\circ}$: $j = \ell_i$ for some 
$1 \le i < k$. Now, 
$$
\frac{T''(T^j x )}{ T'( T^j x ) } 
\cdot \frac{ 1 }{ (T^{m - j })' T^j x } 
= \frac{\beta - 1}{T^{\ell_i} x - \frac12} 
 \frac{1}{ ( T^{\ell_{i+1} 
- \ell_i} )' T^{\ell_i} x }
\cdot \frac{1}{(T^{m - \ell_{i+1}} )' 
T^{\ell_{i+1}} x}.
$$
Since $T^{\ell_i} x \in [ 
\hat{b}_{\ell_{i+1} - \ell_i }, 
\hat{b}_{\ell_{i+1} - \ell_i - 1}
]$,  
\eqref{eq:lb_deriv_kobe} 
combined with Lemma \ref{lem:bn_bound}
yields 
$$
\frac{\beta - 1}{T^{\ell_i} x - \frac12} 
 \frac{1}{ ( T^{\ell_{i+1} 
- \ell_i} )' T^{\ell_i} x }
\le C(\gamma) \frac{ 
\hat{b}_{\ell_{i+1} - \ell_i - 1}
- 
\hat{b}_{\ell_{i+1} - \ell_i } 
 }{ \hat{b}_{\ell_{i+1} - \ell_i } 
- \frac12 }
\le C(\gamma) 
( \ell_{i+1} - \ell_i )^{ -1 },
$$
while \eqref{eq:dist_first} implies 
$$
\frac{1}{(T^{m - \ell_{i+1}} )' 
T^{\ell_{i+1}} x} \le 
C(\gamma, \ell)
 (m - \ell_{i+1})^{ -  1 - 
 \frac{1}{\alpha\beta}  }.
$$
We conclude that 
$$
\frac{T''(T^j x )}{ T'( T^j x ) } 
\cdot \frac{ 1 }{ (T^{m - j })' T^j x } 
\le 
C(\gamma, \ell)
( \ell_{i+1} - \ell_i )^{ -1 }
(m - \ell_{i+1})^{ -  
 1 - \frac{1}{\alpha\beta}  }.
$$

\noindent{\textbf{Case $3^{\circ}$}}: 
$j = \ell_k$. By \eqref{eq:return_lb_3} and 
Lemma \ref{lem:bn_bound}, 
\begin{align*}
  \frac{T''(T^j x )}{ T'( T^j x ) } 
\cdot \frac{ 1 }{ (T^{m - j })' T^j x } 
\le C(\gamma, \ell)
(m - \ell_k)^{- 1}.
\end{align*} 
 
\noindent{\textbf{Case $4^{\circ}$}}: 
$\ell_i < j < \ell_{i+1}$. As in Case $1^{\circ}$ we 
see that
$$
\frac{T''(T^j x )}{ T'( T^j x ) } 
\cdot \frac{ 1 }{ (T^{m - j })' T^j x } 
\le C(\gamma, \ell)
(\ell_{i+1} - j )^{-2} 
(m - \ell_{i+1})^{ - 1 - \frac{1}{\alpha\beta} 
}.
$$

\noindent{\textbf{Case $5^{\circ}$}}: 
$j > \ell_k$.
Since $T^jx \ge b_{m-j+\ell}$, we obtain 
by Lemma \ref{lem:bn_bound} the upper bound 
$$
\frac{T''(T^j x )}{ T'( T^j x ) } 
\le C(\alpha, \ell) (m-j)^{ \frac{1}{\alpha} - 1 }.
$$
Moreover, 
$T^jx \in [b_{m-j+\ell}, b_{m+j}]$. Hence, 
we can use the argument from the proof 
of \eqref{eq:return_lb_3} together 
with Lemma \ref{lem:bn_bound} to obtain 
$$
\frac{ 1 }{ (T^{m - j })' T^j x } 
\le C(\alpha, \ell) (m-j)^{ -\frac{1}{\alpha} - 1 }.
$$
It follows that 
$$
\frac{T''(T^j x )}{ T'( T^j x ) } 
\cdot \frac{ 1 }{ (T^{m - j })' T^j x } 
\le C(\alpha, \ell)
(m - j )^{-2}.
$$

Now, \eqref{eq:dist_second} follows by summing 
over the upper bounds obtained in each case:
\begin{align}
  &\sum_{j=0}^{m-1} \frac{T''(T^j x )}{ T'( T^j x ) } 
  \cdot \frac{ 1 }{ (T^{m - j })' T^j x } \notag\\
  &\le 
  \sum_{j=0}^{\ell_1 - 1} 
  C_1 
(\ell_1 - j)^{-2} (m - \ell_1)^{-
   1 - \frac{1}{\alpha\beta}  }
  \tag{\text{Case $1^{\circ}$}} \\
  &+ 
  \sum_{i=1}^{k-1} C_2
  ( \ell_{i+1} - \ell_i )^{ - 1 }
  (m - \ell_{i+1})^{ - 1 - \frac{1}{\alpha\beta}  }
  \tag{\text{Case $2^{\circ}$}} \\ 
  &+ C_3
  (m - \ell_k)^{- 1}
  \tag{\text{Case $3^{\circ}$}} \\ 
  &+  
  \sum_{i = 1 }^{k-1} 
  \sum_{\ell_i < j < \ell_{i+1}}
  C_4 
  (\ell_{i+1} - j )^{-2} 
  (m - \ell_{i+1})^{ -  1 - 
  \frac{1}{\alpha\beta}  } 
  \tag{\text{Case $4^{\circ}$}} \\
  %\tag{\text{Case $4^{\circ}$}} 
&+ \sum_{j = \ell_{k} + 1}^{m -1 }
C_5 
(m - j )^{-2}
\tag{\text{Case $5^{\circ}$}} \\
&\le C(\gamma,\ell). \notag
\end{align}

\end{proof}

Besides the two bounds in Lemma \ref{lem:distortion_returns}, the 
bound
\eqref{eq:dist_sharp} will be an important tool 
in the proof of the linear response formula.

\section{About cones and invariant measures}\label{sec:cone}

Recall that 
$$
\mathfrak{D} = \{ 
(\alpha, \beta) \in 
(0, 1) \times [1, \infty)
\: : \: 
\alpha \beta < 1
\}.
$$
In this section we show that for 
$\phi \in L^\infty(m)$,
\begin{align*}
    D_i := - 
    \sum_{k=0}^{\infty} \int_0^1
    \phi \cdot 
    \cL_{\gamma}^k[ ( X_{\gamma, i} 
    \cN_{\gamma,i}
    (h_{\gamma})  )'  ] \, dm 
    \quad (i=1,2)
  \end{align*}
is well-defined if  
$\gamma \in \mathfrak{D}$. We start by noting that 
$\cN_{\gamma, i}$ preserves the cone 
constructed in \cite{cui2021invariant}.

%Throughout this section we consider a fixed subset 
%$$
%B = [ \alpha_\ell, \alpha_u ] 
%\times [\beta_\ell, \beta_u] \subset \mathfrak{D}.
%$$

\begin{lem}\label{lem:cone_inv} For 
  $\gamma \in \mathfrak{D}$ define 
  $$
  \cC_a^{(2)}(\gamma)
  = C^2(0, 1] \cap
  \cC_a(\gamma).
  $$
  Suppose $\gamma_u \in \mathfrak{D}$ and 
  $a \ge a_0(\gamma_u)$. Then, for all 
  $\gamma \in \mathfrak{D}$ such that
   $\gamma \le \gamma_u$,
  $h_{\gamma} 
\in \cC_a^{(2)}(\gamma_u)$ together with the following 
inclusions hold:
\begin{align*}
  \cL_{\gamma} 
  (\cC_a^{(2)}(\gamma_u) )
  \subset 
  \cC_a^{(2)}(\gamma_u), \quad 
  \cN_{\gamma, i}
  ( \cC_a^{(2)}(\gamma_u) )
    \subset 
    \cC_{a}^{(2)}(\gamma_u), \quad i \in \{1,2\}.
\end{align*}
\end{lem}

\begin{proof}
  Let $\gamma = (\alpha,\beta) \in \mathfrak{D}$ 
and $\gamma_u = (\alpha_u,\beta_u) \in \mathfrak{D}$  
with $\gamma \le \gamma_u$.
By Lemma \ref{lem:inv_of_lsv_cone}, 
we have  $\cL_{\gamma} 
(\cC_a^{(2)}(\gamma_u) )
\subset 
\cC_a^{(2)}(\gamma_u)$ since 
$\cL_{\gamma} C^2(0,1] \subset C^2(0,1]$ 
clearly holds. Then 
$\cL_\gamma(h_\gamma) = 
h_\gamma \in \cC_a^{(2)}(\gamma_u)$ 
follows by the proof of Theorem 
1.1 in \cite{cui2021invariant}. 
Since $\cL_\gamma(\varphi) = 
\sum_{i=1,2} \cN_{\gamma, i}(\varphi)$, 
it remains to show that $\cN_{\gamma, i}
( \cC_a^{(2)}(\gamma_u) )
  \subset 
  \cC_{a}^{(2)}(\gamma_u)$.

  Let $\varphi \in \cC_a(\gamma_u)$. It is 
  easy to see that $\cN_{\gamma, i}(\varphi) 
  \ge 0$ is decreasing for $i \in \{1,2\}$. 
  Denoting $y_1 = g_{\gamma, 1}(x)$, we have 
  $$
  x^{ \alpha_u + 1 } \cN_{\gamma, 1}(\varphi)(x)
  = y_1^{\alpha_u + 1} \varphi(y_1) 
  \frac{(1 + \xi)^{\alpha_u + 1}}{1 + (\alpha+1)\xi}
  $$
  for $\xi = (2y_1)^\alpha$, from which we see that 
  $x \mapsto 
  x^{ \alpha_u + 1 } 
  \cN_{\gamma, 1}(\varphi)(x)$ 
  is increasing. Denoting  
  $y_2 = 
  g_{\gamma, 2}(x)$,
  we have 
  $$
  x^{ \alpha_u + 1 } \cN_{\gamma, 2}(\varphi)(x)
  = C(\alpha_u, \beta) y_2^{\alpha_u + 1} 
  \varphi(y_2) \cdot y_2^{-\alpha_u - 1} (y_2 - 
  \tfrac12)^{\beta\alpha_u + 1},
  $$
  from which we see that $x \mapsto 
  x^{ \alpha_u + 1 } \cN_{\gamma, 2}(\varphi)(x)$
  is increasing. Finally, using Lemma  
  \ref{lem:inv_of_lsv_cone} we see that 
  $\int_0^x \cN_{\gamma, i}(\varphi) \, dm
  \le \int_0^x \cL_\gamma(\varphi) \, dm \le 
  a m(\varphi) x^{ \frac{1}{\beta_u} 
  - \alpha_u }$.
\end{proof}

Following \cite{baladi2016linear}, we 
define cones involving higher 
order derivatives and establish their 
invariance:

\begin{lem}\label{lem:cone_in_c3} 
  For 
  $b_1 \ge \alpha + 1$, 
  $b_2 \ge b_1$ and $b_3 \ge b_1$, 
  define
  \beqn
\begin{split}
\cC^{(3)}_{b_1,b_2,b_3}
 = 
\{\varphi \in C^3((0,1]) \,:\, 
\varphi(x) \ge 0 , \, 
& |\varphi'(x)| 
\le \tfrac{b_1}{x}\varphi(x),
\, |\varphi''(x)| \le \tfrac{b_2}{x^2}
\varphi(x), \\
& |\varphi'''(x)| \le 
\tfrac{b_3}{x^3} 
\varphi(x), 
\, \forall x \in (0, 1]
\}.
\end{split}
\eeqn
Suppose 
$B = [ \alpha_\ell, \alpha_u ] 
\times [1, \beta_u] \subset \mathfrak{D}$.
There exist constants $\kappa_i > 1$ depending 
only on $B$ such that 
if $b_1 \ge \alpha_u + 1$, 
$b_2 \ge \kappa_1 b_1$ and $b_3 \ge \kappa_2 b_2$, 
then for all $\gamma \in B$, 
$h_\gamma \in \cC^{(3)}_{b_1,b_2,b_3} $ together 
with the following inclusions hold:
$$
\cL_{\gamma} (
\cC^{(3)}_{b_1,b_2,b_3} )
\subset \cC^{(3)}_{b_1,b_2,b_3}, 
\quad 
\cN_{\gamma, i} (
\cC^{(3)}_{b_1,b_2,b_3} )
\subset \cC^{(3)}_{b_1,b_2,b_3}, \quad i 
\in \{1,2\}.
$$
\end{lem}

\begin{proof} The statement concerning 
    $\cN_{\gamma, 1}$ follows from  
  \cite[Appendix B]{baladi2016linear}. 
  The other statements are proved in Appendix 
  \ref{appendix:a}.
\end{proof}

The following simple observation is 
useful when dealing with regular 
functions that do not belong to a cone 
$\cC_a(\gamma)$:

\begin{lem}\label{lem:belong_general} 
    %Let  $\gamma = (\alpha, \beta) \in 
    %B = [ \alpha_\ell, \alpha_u ] 
    %\times [\beta_\ell, \beta_u] \subset \mathfrak{D}$.
    Let $\gamma =(\alpha, \beta) \in \mathfrak{D}$.
   Suppose that $F \in C^2(0, 1]$ is a function 
   such that $m(F) = 0$, and for some $C_i \ge 1$ and $0 < \delta < 1 + \alpha - 
   \tfrac{1}{\beta}$, 
   $$
   | F(x)  | \le C_0 a x^{ - \delta  } 
   \quad \text{and} \quad 
   | F'(x)  | \le C_1 a x^{ - \delta  - 1 }
   $$
   hold for all $x \in (0,1]$.
   Then, $F(x)
   + \lambda x^{\frac{1}{\beta}
   - \alpha - 1}
   \in \cC_a(\gamma)$ if 
   \begin{align}\label{eq:cone_cond}
   a \ge 2 \quad \text{and} \quad 
   \lambda \ge a (C_0 + C_1) \max \left\{ 
    \tfrac{1}{1 + \alpha - 
  \frac{1}{\beta}},
   4 , 
  \tfrac{2( \frac{1}{\beta} - \alpha )}{
         1 - \delta
        }
   \right\}.
   \end{align}
 \end{lem}

 \begin{remark}
  Note that 
  $\lambda x^{\frac{1}{\beta} - \alpha - 1 } 
  \in \cC_a(\gamma)$ if 
  $a \ge 1$.
 \end{remark}

 \begin{proof}[Proof of Lemma \ref{lem:belong_general}] Let us denote 
    \begin{align*}
      \psi(x) = 
      F(x)
    + \lambda x^{\frac{1}{\beta}
    - \alpha - 1}.
    \end{align*}
  \noindent\textbf{Step 1}:
  $\psi \ge 0$. This holds if $\lambda \ge 
  C_0 a$.
  
  \noindent\textbf{Step 2}:
  $\psi$ is 
  decreasing. We see that $\psi'(x) \le 0$ if 
  $$
  \lambda \ge \tfrac{C_1 a}{1 + \alpha - 
  \frac{1}{\beta}}.
  $$
    
  \noindent\textbf{Step 3:} 
  $x^{\alpha + 1}\psi(x)$ is 
  increasing. Let $\varphi(x) 
  = x^{\alpha + 1}\psi(x)$. Then,  
  $$
  \lambda \ge  a 4 ( C_0 + C_1) 
  \implies \varphi'(x) \ge 0.
  $$
  
  \noindent\textbf{Step 4:} 
  $\int_0^x \psi(t) \, dt \le 
    a
     x^{ \frac{1}{\beta} 
     - \alpha } m(\psi)$. 
     Since $m(F) = 0$, we have 
     $m(\psi) = 
     \int_0^1 \lambda x^{\frac{1}{\beta}
     - \alpha - 1} \, dx
     = \lambda ( \tfrac{1}{\beta} 
     - \alpha)^{-1}$. Hence, 
     \begin{align*}
     \int_0^x \psi(t) \, dt 
     &\le m(\psi) a x^{\frac{1}{\beta} - 1 }
     \biggl( 
     \tfrac{C_0 ( \frac{1}{\beta} - 
     \alpha) }{\lambda (1 - \delta)} 
     x^{ 1 - \delta + \alpha - \frac{1}{\beta} }
     + \tfrac{1}{a}
     \biggr) \\
     &\le 
     m(\psi) a x^{\frac{1}{\beta} - 1 }
     \biggl( 
        \tfrac{C_0 ( \frac{1}{\beta} - 
        \alpha) }{\lambda (1 - \delta)} 
     + \tfrac{1}{2}
     \biggr).
     \end{align*}
     It follows that 
     \begin{align*}
        \lambda \ge 
        \tfrac{2C_0
        ( \frac{1}{\beta} - \alpha )
        }{
         (1 - \delta)
        }
        \implies 
        \int_0^x \psi(t) \, dt \le 
    a
     x^{ \frac{1}{\beta} 
     - \alpha } m(\psi).
     \end{align*}
  
  \noindent\textbf{Conclusion:} 
  We obtain 
  $$
  \eqref{eq:cone_cond} \implies 
  F(x)
  + \lambda x^{\frac{1}{\beta}
  - \alpha - 1}
  \in \cC_a(\gamma)
  $$
  by combining  
  the bounds established in Steps 1 through 4.
  \end{proof}

We record 
in the following 
two lemmas some basic 
properties of the 
functions $X_{\gamma, i}$ 
($i=1,2$) and their 
derivatives, 
which are easy to verify. 

%In what follows, we write 
%$a \lesssim_\theta b$ if there 
%exists a constant $C > 0$ depending 
%on the parameter $\theta$ such that 
%$a \le C b$.

\begin{lem}[Properties of 
  $X_{\gamma, 1}$]\label{lem:x_alpha}
  The function $(\alpha, x) \mapsto X_{\gamma, 1} 
  (x)$ is $C^2$ on $(0, \infty) \times (0, 1]$. 
  If $B = [ \alpha_\ell, \alpha_u ] 
\times [1, \beta_u] \subset \mathfrak{D}$, 
  there exist  $C_i = C_i(B) 
  > 0$ such that for all $x \in (0,1]$ and all $\gamma \in B$, 
\begin{align*}
&|X_{\gamma, 1}(x)| \le C_0
x^{1+\alpha}( 1 - \log (x) ), 
\quad |X_{\gamma, 1}'(x)|
\le C_1 x^\alpha ( 1 - \log(x)), \\ 
&|X_{\gamma, 1}''(x)|
\le C_2 x^{\alpha-1} 
( 1 - \log(x)).
\end{align*}
\end{lem}

\begin{lem}[Properties of $X_{\gamma, 2}$]\label{lem:x_beta}
    The function $(\beta, x) \mapsto X_{\gamma, 2}$ 
    is $C^2$ on $(0, \infty) \times (0,1]$.
    If $B = [ \alpha_\ell, \alpha_u ] 
\times [1, \beta_u] \subset \mathfrak{D}$, 
  there exist  $C_i = C_i(B) 
  > 0$ such that for all $x \in (0,1]$ and all $\gamma \in B$
  \begin{align*}
  |X_{\gamma, 2}(x)| \le 
  C_0 \log( x^{-1} ) x, \quad
  |X_{\gamma, 2}'(x)| \le
  C_1 [ 1 - \log(x)], 
  \quad
  |X_{\gamma, 2}''(x)| 
  \le C_2 \tfrac{1}{x}. 
\end{align*}
\end{lem}

Next, we verify that $D_i$ is well-defined 
for $i \in \{1,2\}$.

\begin{lem}\label{lem:D_well_defined} Suppose 
that $\gamma \in \mathfrak{D}$ and that 
$\psi \in \cC_a(\gamma) \cap \cC_{b_1, b_2, b_3}^{(3)}$ for $a \ge a_0(\gamma)$ and $b_1,b_2,b_3$ as in 
Lemma \ref{lem:cone_in_c3}.
Let  
$i \in \{1,2\}$. Then, for all $\delta \in (0,1)$ 
there exist $\psi_{1}, \psi_2 \in 
\cC_a(\delta, \beta)$ 
such that 
\begin{align}\label{eq:decomp_di}
(X_{\gamma, i} \cN_{\gamma, i}
    (\psi)  )' = 
    \psi_{1} - \psi_2.
\end{align}
In particular, if $\gamma_1, \gamma \in 
B = [ \alpha_\ell, \alpha_u ] 
\times [1, \beta_u] \subset \mathfrak{D}$,
then 
for all $\phi \in L^\infty(m)$, 
\begin{align*}
    S_i = - \sum_{k=0}^{\infty} \int_0^1
    \phi \cdot \cL_{\gamma_1}^k[ ( 
        X_{\gamma, i} \cN_{\gamma, i}
    ( \psi )  )'  ] \, dm, \quad i \in \{1,2\},
  \end{align*}
is absolutely summable.
\end{lem}

\begin{remark}
  Since $h_\gamma \in \cC_a(\gamma)$, it 
  follows that $D_i$ is absolutely summable  
  for $i \in \{1,2\}$.
\end{remark}

\begin{proof}[Proof of 
  Lemma \ref{lem:D_well_defined}] For $i = 1$, writing 
    $u =
    ( X_{\gamma,1} \cN_{\gamma,1}
        (\psi)  )'$ we have 
    $m(u) = 0$. Moreover, by Lemmas 
    \ref{lem:x_alpha} 
    and \ref{lem:cone_in_c3}, there 
    exist $C_i = C_i(\gamma, a, b_1, b_2) > 0$ 
    such that 
    $$
    |u(x)| \le C_1 x^{\frac{1}{\beta} - 1} 
    (1 - \log (x)) 
    $$
    and 
    $$
    |u'(x)| \le C_2 x^{\frac{1}{\beta} - 2} 
    (1 - \log (x)).
    $$
    Hence, in Lemma \ref{lem:belong_general}
    we may take $\alpha$ arbitrary small. That 
    is, for any $\delta \in (0,1)$
    there exists $\lambda = 
    \lambda(\gamma, \delta, a, b_1, b_2) 
    > 0$ 
    such that 
    $$
    u(x) + \lambda x^{\frac{1}{\beta} 
    - \delta -1 } \in \cC_a(\delta, \beta),
    $$
    if $a \ge 2$. Since 
    $\lambda x^{\frac{1}{\beta} 
    - \delta -1 } \in \cC_a(\delta, \beta)$, 
    the decomposition \eqref{eq:decomp_di} 
    follows.

    For $i = 2$, we define 
    $$
    v =
    ( X_{\gamma, 2} \cN_{\gamma, 2}
        (\psi)  )'
    = X_{\gamma, 2}'\cN_{\gamma, 2}
    (\psi) + X_{\gamma, 2}\cN_{\gamma, 2}
    (\psi)'.
    $$
    Then $m(v) = 0$. The upper bound  
    $\psi(g_{\gamma, 2}(x)) 
    \le \psi(\tfrac12) \le 2^{\alpha+1}$ 
    combined with 
    Lemma \ref{lem:x_beta} yields 
    \begin{align*}
      | X_{\gamma, 2}'\cN_{\gamma, 2}
    (\psi)(x) |
    &\le (1 - \log(x)) \cdot g_{\gamma, 2}'(x) 
    \psi(g_{\gamma, 2}(x)) \\ 
    &\le (1 - \log(x)) \psi( \tfrac12 ) 
    \frac{1}{ T_\gamma'( g_{\gamma, 2}(x)  )} \\ 
    &\le C_3 x^{\frac{1}{\beta} - 1} 
    (1 - \log (x)),
    \end{align*}
    and, by Lemmas \ref{lem:x_beta} 
    and \ref{lem:cone_in_c3},
    \begin{align*}
      |X_{\gamma, 2}\cN_{\gamma, 2}(x)
      (\psi)'| \le C_4 (1 - \log(x)) x 
      \cdot a x^{-1} \cN_{\gamma, 2}(\psi)(x) 
      \le C_4 x^{\frac{1}{\beta} - 1} 
      (1 - \log (x)).
    \end{align*}
    Hence, $|v(x)| \le C_5 x^{\frac{1}{\beta} - 1} 
    (1 - \log (x))$. Similarly, using  
    Lemmas \ref{lem:x_beta} 
    and \ref{lem:cone_in_c3}, we see that 
    $|v'(x)| \le C_6 x^{\frac{1}{\beta} - 2} 
    (1 - \log (x))$.
    Again 
    the decomposition \eqref{eq:decomp_di} 
    follows 
    by Lemma \ref{lem:belong_general}.

    The decomposition \eqref{eq:decomp_di} 
    combined with the 
    first part of Lemma \ref{lem:memory_loss} 
    yields the 
    upper bound 
    $$
    \Vert \cL_{\gamma_1}^k[ ( 
        X_{\gamma, i} \cN_{\gamma, i}
    (\psi)  )'  ] \Vert_{L^1(m)} 
    \le 
    C ( m(\psi_1) + m(\psi_2) ) 
    k^{-\gamma_*},
    $$
    for some $C > 0$ determined by 
    $a, B, \gamma, \delta$ and 
    $$
    \gamma_* =  \tfrac{1}{\alpha_u} 
    (  \tfrac{1}{\beta_u} - \delta ).
    $$
    For sufficiently small $\delta > 0$ 
    we have $\tfrac{1}{\alpha_u} 
    (  \tfrac{1}{\beta_u} - \delta ) > 1$. It 
    follows that $S_i$ is absolutely summable 
    for $i \in \{1,2\}$.
\end{proof}

\section{Regularity properties of  
$\cL_{\gamma}(\varphi)
$}\label{sec:regularity}

For sufficiently regular $\varphi : [0,1] \to \bR$,  
we analyze properties of the partial derivatives 
of $\gamma \mapsto 
\cL_{\gamma} (\varphi) (x)$. Since the first 
branch of $T_\gamma$ does not depend on 
$\beta$ and the second branch does not 
depend on $\alpha$, we do not have to consider 
mixed derivatives.
As before, given 
$\alpha$ and $\beta$ we denote $\gamma = (\alpha, 
\beta)$.

\begin{lem}\label{
  lem:regularity_inverse_branch
} The functions 
  %$G_1$ and $G_2$ are 
  $(\alpha, y) 
  \mapsto g_{\gamma, 1}(y)$ 
  and $(\beta, y) \mapsto g_{\gamma, 2}(y)$ 
  are both 
  $C^4$ on $(0,\infty) \times (0,1)$, 
  and for all $x \in (0,1)$, $\alpha, \beta > 0$, 
  \begin{align}\label{eq:partial_g_1}
    %\partial_1 G_1(\alpha, y) 
    %= 
    \partial_i g_{\gamma, i}(x)
    = - \frac{X_{\gamma, i}(x)}{ f_{\gamma, i}' 
    (g_{\gamma, i} (x) ) },
    \quad i \in \{1,2\}.
  \end{align}
\end{lem}

\begin{proof}
  The result is an 
  immediate consequence of  
  the implicit function theorem 
  applied 
  with $F_1( (\alpha, x), y ) = f_{\gamma, 1} (y) - x$
  and $F_2( (\beta, x), y ) = f_{\gamma, 2} (y) - x$.
\end{proof}

\begin{lem}\label{lem:first_partial_L} 
  Suppose that $\varphi \in C^1((0,1], \bR)$. Then, 
for all $x \in (0,1)$, $\alpha, \beta > 0$, 
and $i \in \{1,2\}$,
\begin{align*}
  \partial_i \cL_{\gamma}(\varphi)(x)
  %= \partial_\alpha \cN_{\alpha, 1}(\varphi)
  &= - X_{\gamma, i} ' \cN_{\gamma, i}(\varphi) (x)
  - X_{\gamma, i} \cN_{\gamma, i}( \varphi' / 
  f_{\gamma, i}' )(x)
  + X_{\gamma, i} \cN_{\gamma, i} 
  (\varphi f_{\gamma, i}'' / (f'_{\gamma, i})^2 )(x)
  \\
  &= - ( X_{\gamma, i} \cN_{\gamma, i}(\varphi) )'(x).
\end{align*}
In particular, 
$m( \partial_i \cL_{\gamma} (\varphi) ) = 0$ 
for $i=1,2$.
\end{lem}

\begin{proof}
  These formulas are from \cite{baladi2014linear}; for 
  completeness, we give 
  a proof  in the case of 
  $\partial_2 
  \cL_{\gamma}(\varphi)(x)$, the proof 
  for $\partial_1 
  \cL_{\gamma}(\varphi)(x)$ being the same. As 
  the second branch of $T_\gamma$ is independent 
  of $\alpha$, in the following we will 
  denote $f_{\beta,2} = f_{\gamma, 2}$, 
  $g_{\beta, 2} = g_{\gamma, 2}$, 
  $X_{\beta, 2} = X_{\gamma, 2}$, and 
  $\cN_{\beta,2} = \cN_{\gamma,2}$, omitting 
  $\alpha$ from the subscripts.

  First, we verify that 
  \begin{align}\label{eq:partial_g_prime_1}
  \partial_2 g_{\beta, 2}'(x) 
  = - \frac{ X'_{\beta, 2}(x)}{ 
    f_{\beta,2}'(g_{\beta, 2} (x) ) } 
    + X_{\beta, 2}
     \frac{ f_{\beta, 2}''(g_{\beta, 2}(x)) }{
      (f_{\beta, 2}' ( g_{\beta, 2}(x) ) )^3
    }.
  \end{align}
  For $\delta > 0$ sufficiently small we have 
  $g_{\beta + \delta,2}'(x) - 
  g_{\beta,2}'(x) = I + II$, where 
  \begin{align*}
    I = \frac{ f'_{\beta, 2}( g_{\beta,2}(x) )
    - f'_{\beta + \delta, 2}( g_{\beta,2}(x) ) }{
      f'_{\beta + \delta,2}( g_{\beta,2}(x) )
      f'_{\beta,2}( g_{\beta,2}(x) )
    } \quad \text{and} \quad 
    II = \frac{ f'_{\beta + \delta,2}( 
      g_{\beta,2}(x) )
    - f'_{\beta + \delta,2}( 
      g_{\beta + \delta,2}(x) ) }{
      f'_{\beta + \delta,2}( 
        g_{\beta,2}(x) )f'_{\beta + \delta,2}( 
          g_{\beta + \delta,2}(x) )
    }. 
  \end{align*}
  By Taylor's theorem, for any $y \in (\tfrac12, 1)$, 
  \begin{align*}
    f_{\beta + \delta,
    2} ' (y) = f_{\beta,1} ' (y) + 
    \partial_2 f'_{\beta,2} (y)
    \cdot \delta + O(\delta^2),
  \end{align*}
  where $\partial_2 f'_{\beta,2} (y) = 
  (\partial_2 f_{\beta,2})' (y) 
  = v_{\beta,2}'(y)$, since 
  $(\beta, y) \mapsto f_{\beta,2}(y)$ is $C^2$.
  It follows that 
  \begin{align*}
    \lim_{\delta \to 0 }
    \frac{I}{\delta} = - 
    \frac{v_{\beta,2}'(g_{\beta,2}(x)) 
    }{ ( f_{\beta,2}'(g_{\beta,2} (x) ) )^2 } 
    = - \frac{ X'_{\beta,2}(x)}{ 
      f_{\beta,2}'(g_{\beta,2} (x) ) }.
  \end{align*}
  By a similar computation, using Taylor's 
  theorem together with
   $g_{\beta + \delta, 2}(y) - g_{\beta, 2}(y) 
   = O(\delta)$,
  we see that 
  \begin{align*}
    \lim_{\delta \to 0}
    \frac{II}{\delta} 
    = \frac{X_{\beta,2}}{ f_{\beta,2}'(g_{\beta,2}(x)) }
    \cdot \frac{ f_{\beta,2}''(g_{\beta,2}(x)) }{
      (f_{\beta,2}' ( g_{\beta,2}(x) ) )^2
    }.
  \end{align*}
  Hence, \eqref{eq:partial_g_prime_1} holds. Using 
  \eqref{eq:partial_g_1}, 
  we obtain 
  \begin{align*}
    \partial_2 \cL_{\alpha,\beta}(\varphi)(x)
    &= \partial_2 \cN_{\beta,2}(\varphi)(x)
    = \partial_2 g_{\beta,2}'(x) 
    \cdot \varphi(g_{\beta,2}
    (x))
    +  g_{\beta,2}'(x) \varphi'(g_{\beta,2}(x)) 
    \cdot 
    \partial_\alpha g_{\beta,2} (x) \\
    &= \partial_2 g_{\beta,2}'(x) 
    \cdot \varphi(g_{\beta,2}
    (x))
    - X_{\beta,2} \frac{\varphi'( g_{\beta,2}(x) )}
    {( f_{\beta,2}' ( g_{\beta,2}(x) ) )^2},
  \end{align*}
where the last term may be written as 
\begin{align*}
  X_{\beta,2} \frac{\varphi'( g_{\beta,2}(x) )}
    {( f_{\beta,2}' ( g_{\beta,2}(x) ) )^2}
    = X_{\beta,2} \cN_{\beta,2}( 
      \varphi' / f_{\beta,2}' )(x),
\end{align*}
and, using \eqref{eq:partial_g_prime_1}, the 
first term may be written as 
\begin{align*}
\partial_2 g_{\beta,2}'(x) \cdot 
\varphi(g_{\beta,2}
(x)) = - X_{\beta,2} ' \cN_{\beta,2}(\varphi) (x)
+ X_{\beta,2} \cN_{\beta,2} 
(\varphi f_{\beta,2}'' / (f'_{\beta,2})^2 )(x).
\end{align*}
\end{proof}

\begin{lem}\label{lem:partial_second_L}
  If $\varphi \in C^3((0,1])$, then
  $(\alpha, x) \mapsto 
  X_{\gamma, 1} \cN_{\gamma, 1}(\varphi)(x)$
  and 
  $(\beta, x) \mapsto 
  X_{\gamma, 2} \cN_{\gamma, 2}(\varphi)(x)$
  are $C^3$ on $(0, \infty) \times (0,1)$, 
  and for $i \in \{1,2\}$,
  \begin{align}\label{eq:second_partial_L_1}
  &\partial_i^2 \cL_{\gamma}(\varphi)(x) \notag
    \\&= ( \partial_i X_{\gamma, i}  
    \cN_{\gamma, i}(\varphi) )'(x)
    + X_{\gamma, i}' ( X_{\gamma, i} 
    \cN_{\gamma, i}(\varphi) )'(x)
    + X_{\gamma, i} ( X_{\gamma, i} 
    \cN_{\gamma, i}(\varphi) )''(x).
  \end{align}
Moreover, $m( \partial_i^2 
\cL_{\gamma}(\varphi) )
= 0$.
\end{lem}

\begin{proof} Using Lemma 
  \ref{ lem:regularity_inverse_branch } 
  we see that 
  $(\alpha, x) \mapsto 
  X_{\gamma, 1} \cN_{\gamma, 1}(\varphi)(x)$
  and 
  $(\beta, x) \mapsto 
  X_{\gamma, 2} \cN_{\gamma, 2}(\varphi)(x)$
  are $C^3$. Then, by Lemma 
  \ref{lem:first_partial_L}, 
  \begin{align*}
    \partial_i^2 \cL_{\gamma}(\varphi)(x) 
    &=
     [ \partial_i ( 
      X_{\gamma, i} \cN_{\gamma, i}(\varphi) )]'(x) \\
    &= 
    ( \partial_i X_{\gamma, i}  
    \cN_{\gamma, i}(\varphi) )'(x) 
    + X'_{\gamma, i}  \partial_i \cN_{\gamma, i}
    (\varphi)(x)
    + X_{\gamma, i}  ( \partial_i 
    \cN_{\gamma, i}(\varphi) )' (x),
  \end{align*}
  and now the desired formula for 
  $\partial_i^2 \cL_{\gamma}(\varphi)(x)$
  follows by another application of Lemma 
  \ref{lem:first_partial_L}. Finally, 
  we have 
  $m( \partial_i^2 
  \cL_{\gamma}(\varphi) )
  = \partial_i 
  m( \partial_i 
  \cL_{\gamma}(\varphi) )
  = 0$ by Lemma \ref{lem:first_partial_L}, 
  where the use of Leibniz integral 
  rule is justified by 
  \eqref{eq:bound_on_d2_1} 
  and \eqref{eq:bound_on_d2_2} below.
\end{proof}

\begin{lem}\label{lem:bounds_on_second_1}
  Suppose that 
  $\gamma = (\alpha, \beta) \in 
  B = [\alpha_\ell, \alpha_u] \times 
  [1, \beta_u] \subset \mathfrak{D}$.
  Let $\varphi \in \cC_{b_1,b_2,b_3}^{(3)}
  \cap C_{a}(\gamma_u)$, $m(\varphi) = 1$,  where 
  $a \ge a_0(\gamma_u)$ and $b_i$ are 
  as in Lemma \ref{lem:cone_in_c3}.
  Then, for any
  $0 < \widetilde{\alpha} 
  < \alpha$, there exist $C_i > 0$ depending only on 
  $B, a, b_1, b_2, b_3, \widetilde{\alpha}$, 
  such that 
\begin{align}\label{eq:bound_on_d2_1}
| \partial_1^2 \cL_{\gamma}(\varphi)(x)  | 
\le C_0 x^{-\kappa} 
\quad \text{and} 
\quad | ( \partial_1^2 \cL_{\gamma}(\varphi))' (x)  
| \le C_1 x^{-\kappa + 1}, 
\end{align}
where 
$$
\kappa <  1 + (\alpha_u - \widetilde{\alpha})
- \tfrac{1}{\beta_u} 
.
$$
In particular, there exist $\varphi_i \in 
\cC_a( \alpha_u - \widetilde{\alpha}, 
\beta_u)$ 
such that 
\begin{align}\label{eq:decomp_second_diff}
\partial_1^2 \cL_\gamma(\varphi) 
= \varphi_1 - \varphi_2 \quad 
\text{and} \quad 
\Vert \varphi_i \Vert_{L^1(m)} 
\le C_i',
\end{align}
where $C_i' > 0$ depend 
only on $B, a, b_1, b_2, b_3, \widetilde{\alpha}$.
\end{lem}

\begin{proof} To obtain 
  the desired upper bound on 
  $\partial_1^2 \cL_{\gamma}(\varphi)(x)$
  we have 
  to control each of the three terms 
  in \eqref{eq:second_partial_L_1}. 
  Let us denote them respectively by 
  $E_1$, $E_2$ and $E_3$. By a direct 
  computation, we see that 
  \begin{align*}
  | \partial_1 X_{\gamma, 1}(x)  | 
  \le C(B) x^{1 + \alpha} 
  ( 1 + \log(x^{-1}))^2
  \quad \text{and} \quad
  | \partial_1 X_{\gamma, 1}'(x) | \le C(B) 
  x^{\alpha} 
  (1 + \log(x^{-1}))^2.
  \end{align*}
  These upper bounds together with Lemmas 
  \ref{lem:x_alpha} and \ref{lem:cone_in_c3} 
  yield 
  \begin{align*}
  |E_1| &= |( \partial_1 X_{\gamma, 1}  
  \cN_{\gamma, 1}(\varphi) )'(x)| 
  \le C_1(B, a, b_1) x^{
  \frac{1}{\beta_u} -  
  \alpha_u - 1 + \alpha} 
  (  1 + \log(x^{-1}) )^2,
\end{align*}
\begin{align*}
    |E_2| &= |X_{\gamma, 1}' ( X_{\gamma, 1} 
  \cN_{\gamma, 1}(\varphi) )'(x)| \le 
  C_2(B, a, b_1) x^{
  \frac{1}{\beta_u} -  
  \alpha_u - 1 + 2\alpha} 
  (  1 + \log(x^{-1}) )^2,
\end{align*}
\begin{align*}
  |E_3| &= |X_{\gamma, 1} ( X_{\gamma, 1} 
  \cN_{\gamma, 1}(\varphi) )''(x)| \le 
C_3(B, a, b_1,b_2) x^{
\frac{1}{\beta_u} -  
\alpha_u - 1 + 2\alpha} 
(  1 + \log(x^{-1}) )^2.
\end{align*}
Hence the first inequality in  
\eqref{eq:bound_on_d2_1} holds. For the 
second inequality we have to differentiate 
each of the three terms $E_i$. We have 
$$
E_1' = ( \partial_1 X_{\gamma, 1}  
\cN_{\gamma, 1}(\varphi) )''(x) 
= \partial_1 X_{\gamma, 1}'' \cN_{\gamma, 1}(\varphi)
+ 2 \partial_1 X_{\gamma, 1}' \cN_{\gamma, 1}(\varphi)'
+ \partial_1 X_{\gamma, 1} \cN_{\gamma, 1}(\varphi)''.
$$
Using 
$
| \partial_1 X_{\gamma, 1}''(x)| \le 
C(B)  x^{\alpha-1} ( 1 + \log(x^{-1}))^2,
$
Lemma \ref{lem:cone_in_c3}, 
and the bounds on $\partial_1 X_{\gamma, 1}(x)$ 
and $ \partial_1 X_{\gamma, 1}'(x)$ above 
we 
see that  
$$
|E_1'| \le C_1(B, a, b_1, b_2) 
x^{ \frac{1}{\beta_u} -  
\alpha_u - 2 + \alpha} ( 1 + \log(x^{-1}))^2.
$$
Using Lemmas \ref{lem:x_alpha} and \ref{lem:cone_in_c3} 
we obtain $|E_j'| \le C_i(B, a, b_1, b_2, b_3) 
x^{ \frac{1}{\beta_u} -  
\alpha_u - 2 + \alpha} ( 1 + \log(x^{-1}))^2$ 
for $j=2,3$, so that the second inequality in 
\eqref{eq:bound_on_d2_1} holds. Now, 
\eqref{eq:decomp_second_diff} is an immediate 
consequence of Lemma \ref{lem:belong_general} 
and the first inequality in 
\eqref{eq:bound_on_d2_1}.
\end{proof}

\begin{lem}\label{lem:bounds_on_second_2}
  Suppose that 
  $ \gamma \in 
  B = [\alpha_\ell, \alpha_u] \times 
  [1, \beta_u] \subset \mathfrak{D}$.
  Let $\varphi \in \cC_{b_1,b_2,b_3}^{(3)}
  \cap C_{a}(\gamma_u)$, $m(\varphi) = 1$,  where 
  $a \ge a_0(\gamma_u)$ and $b_i$ are 
  as in Lemma \ref{lem:cone_in_c3}.
  Then for any $0 < \widetilde{\alpha}_u < \alpha_u$ 
  there exist $C_i > 0$ depending only on 
  $B, a, b_1, b_2, b_3, \widetilde{\alpha}_u$
  such that 
\begin{align}\label{eq:bound_on_d2_2}
| \partial_2^2 \cL_{\gamma}(\varphi)(x)  | 
\le C_0 x^{-\kappa} 
\quad \text{and} 
\quad | ( \partial_2^2 \cL_{\gamma}(\varphi))' (x)  
| \le C_1 x^{-\kappa + 1}, 
\end{align}
where 
$$
\kappa <  1 + \widetilde{\alpha}_u
- \tfrac{1}{\beta_u} 
.
$$
In particular, there exist $\varphi_i \in 
\cC_a( \widetilde{\alpha}_u, 
\beta_u)$ 
such that 
\begin{align}\label{eq:decomp_second_diff_2}
\partial_2^2 \cL_\gamma(\varphi) 
= \varphi_1 - \varphi_2 \quad 
\text{and} \quad 
\Vert \varphi_i \Vert_{L^1(m)} 
\le C_i',
\end{align}
where $C_i' > 0$ depend only on 
$B, 
  a, b_1, b_2, b_3, \widetilde{\alpha}_u$.
\end{lem}

\begin{proof} Denote 
  each of the three terms 
  in \eqref{eq:second_partial_L_1}
  respectively by 
  $E_1$, $E_2$ and $E_3$.
  The upper bounds 
  \begin{align*}
    &|\cN_{\gamma, 2}(\varphi)(x)| \le 
    C_1(B) x^{ \frac{1}{\beta_u} - 1 }, \\ 
  &| \partial_2 X_{\gamma, 2}(x)  | 
  \le C_2(B) \log(x^{-1}) x
  \quad \text{and} \quad
  | \partial_2 X_{\gamma, 2}'(x) | \le C_3(B) 
  (1 + \log(x^{-1}))
  \end{align*}
  combined with Lemmas 
  \ref{lem:x_beta} and \ref{lem:cone_in_c3} 
  yield 
  \begin{align*}
  |E_1| &= |( \partial_2 X_{\gamma, 2}  
  \cN_{\gamma, 2}(\varphi) )'(x)| 
  \le C_1(B, a, b_1) x^{
  \frac{1}{\beta_u} -  
   1 } 
  (  1 + \log(x^{-1}) ),
\end{align*}
\begin{align*}
    |E_2| &= |X_{\gamma, 2}' ( X_{\gamma, 2} 
  \cN_{\gamma, 2}(\varphi) )'(x)| \le 
  C_2(B, a, b_1) x^{
  \frac{1}{\beta_u}   
   - 1 } 
  (  1 + \log(x^{-1}) )^2,
\end{align*}
\begin{align*}
  |E_3| &= |X_{\gamma, 2} ( X_{\gamma, 2} 
  \cN_{\gamma, 2}(\varphi) )''(x)| \le 
  C_3(B, a, b_1, b_2) x^{
    \frac{1}{\beta_u} - 1 } 
    (  1 + \log(x^{-1}) )^2.
\end{align*}
Hence, the first inequality in  
\eqref{eq:bound_on_d2_2} holds. 
To obtain the 
second inequality in \eqref{eq:bound_on_d2_2}
we differentiate 
$E_1$, 
$$
E_1' = ( \partial_2 X_{\gamma, 2}  
\cN_{\gamma, 2}(\varphi) )''(x) 
= \partial_2 X_{\gamma, 2}'' \cN_{\gamma, 2}(\varphi)
+ 2 \partial_2 X_{\gamma, 2}' 
\cN_{\gamma, 2}(\varphi)'
+ \partial_2 X_{\gamma, 2} \cN_{\gamma, 2}(\varphi)''
$$
and observe that 
$
| \partial_2 X_{\gamma, 2}''(x)| \le C x^{-1}
$.
Using Lemma \ref{lem:cone_in_c3}, 
together with the bounds on 
$\cN_{\gamma, 2}(\varphi)(x)$,
$\partial_2 X_{\gamma, 2}(x)$, 
$ \partial_2 X_{\gamma, 2}'(x)$, and  
$\partial_2 X_{\gamma, 2}''$ above 
we arrive at the upper bound 
$$
|E_1'| \le C_1(B, a, b_1, b_2) 
x^{ \frac{1}{\beta_u} - 2} ( 1 + \log(x^{-1}))^2.
$$
By Lemmas \ref{lem:x_beta} 
and \ref{lem:cone_in_c3},
$|E_j'| \le C_i(B, a, b_1, b_2, b_3) 
x^{ \frac{1}{\beta_u} -  2} ( 1 + \log(x^{-1}))^2$ 
for $j=2,3$, so that the second inequality in 
\eqref{eq:bound_on_d2_2} holds. 
Now, 
\eqref{eq:decomp_second_diff_2} 
follows 
by Lemma \ref{lem:belong_general} combined 
with the first inequality in \eqref{eq:bound_on_d2_2}.
\end{proof}

\section{Proof of Theorem \ref{thm:main}
for $\phi \in L^\infty$
}
\label{sec:proof_main}
%%%%%%%%%%%%%%%%%%%%%%%%%%%%%%%%%%%%%%%
%%%%%%%%%%%%%%%%%%%%%%%%%%%%%%%%%%%%%%%
%%%%%%%%%%%%%%%%%%%%%%%%%%%%%%%%%%%%%%%
%%%%%%%%%%%%%%%%%%%%%%%%%%%%%%%%%%%%%%%
%%%%%%%%%%%%%%%%%%%%%%%%%%%%%%%%%%%%%%%

In this section we prove Theorem 
\ref{thm:main} for $\phi \in L^\infty(m)$.
The generalization to $L^q$ observables 
will be carried out in Appendix \ref{appendix:b}. 
The proof for $\phi \in L^\infty(m)$
is based 
on the approach of 
\cite{baladi2016linear}
and relies on results established  
in the previous sections. 

Let $\gamma = (\alpha, \beta) \in \mathfrak{D}$. 
The task 
is to show that 
$$
\cR_\phi(\gamma + \delta) -  
\cR_\phi(\gamma ) - D \cdot \delta
= o( \Vert \delta \Vert )
$$
as $\delta = 
(\delta_1, \delta_2) \to 0$, where 
$D = (D_1, D_2)$ and, for $i \in \{1,2\}$,
\begin{align*}
  D_i = - \sum_{k=0}^{\infty} \int_0^1
  \phi \cdot \cL_{\gamma}^k[ ( X_{\gamma, i} 
  \cN_{\gamma,i}
  (h_{\gamma})  )'  ] \, dm.
\end{align*}
By Lemma \ref{lem:D_well_defined}, 
$D_1$ 
and $D_2$ are well-defined. Let $\gamma_u = (\alpha_u, \beta_u) = (\alpha, \beta) + (\delta_*, \delta_*)$ 
and $\alpha_\ell = \alpha - \delta_*$, where $\delta_* > 0$ is sufficiently small such that 
$$
B := 
[ \alpha_\ell, \alpha_u ] \times [1, \beta_u]
\subset \mathfrak{D}.
$$
In the rest of the proof we only consider $\delta$ with $\gamma + \delta \in B$.
%Then $[\alpha - 5 \Vert \delta \Vert, 
%\alpha + 5 \Vert \delta \Vert] \subset (\alpha_\ell, 
%\alpha_u)$ and $\beta + 5 \Vert \delta \Vert <
%\beta_u$ hold 
%whenever $\Vert \delta \Vert$ is sufficiently 
%small,
%and we shall assume this in the rest of the 
%proof.

%We assume, without loss of generality, 
%that $\mu_\gamma(\phi) = 0$. Then, 
By Lemma \ref{lem:memory_loss},
\begin{align*}
  \cR_\phi(\gamma + \delta) -  
  \cR_\phi(\gamma )
  = \int_0^1 \phi \circ T^k_{\gamma + \delta} \, dm 
  - \int_0^1 \phi \circ T^k_\gamma \, dm
  + O ( \Vert \phi  \Vert_\infty k^{  1- \frac{1}{\alpha_u 
  \beta_u}   })
\end{align*}
for all $k \ge 1$, 
where the constant in the error term 
depends only on $\gamma$ and $\delta_*$. Let $\xi > 0$. 
Then there exists a constant 
$\mathfrak{c} = \mathfrak{c}(\gamma, \delta_*) > 0$, 
such that for 
$$
k \ge k_0(\xi, \delta, B) 
:= \lceil \mathfrak{c} \Vert \delta \Vert^{- (1 + \xi) 
\frac{\alpha_u\beta_u}{1 - \alpha_u\beta_u}} \rceil
$$
we have  
\begin{align*}
\frac{ \cR_\phi(\gamma + \delta) -  
\cR_\phi(\gamma ) - D \cdot \delta}{ \Vert \delta  \Vert}
= P_k + O( \Vert \phi \Vert_\infty 
\Vert \delta  \Vert^\xi ) ,
\end{align*}
where
\begin{align}\label{eq:p_k}
P_k 
:= \frac{
    \int_0^1 \phi \circ T^k_{\gamma + \delta} \, dm 
  - \int_0^1 \phi \circ T^k_\gamma \, dm
  - D \cdot \delta
}{  \Vert \delta \Vert }.
\end{align}
It remains to show that 
$P_k \to 0$ as $\delta \to 0$.

For  $\varphi \in C^3(0,1]$ and $x \in (0,1]$,
Lemma \ref{lem:partial_second_L} combined 
with Taylor's theorem
gives the formula
\begin{align*}
(\cL_{\gamma + \delta} - \cL_\gamma)(\varphi)(x)
&= \partial_1 \cL_\gamma (\varphi) (x) \delta_1
+ \partial_2 \cL_\gamma (\varphi) (x) \delta_2
\\
&+ \delta_1^2 \int_0^1 (1 - t) \partial_1^2 
\cL_{\gamma + \delta t} (\varphi) (x)  \, dt
+ \delta_2^2 \int_0^1 (1 - t) \partial_2^2 
\cL_{\gamma + \delta t} (\varphi) (x) \, dt.
\end{align*}
Thus, by writing
\begin{align}\label{eq:decomp_diff_int}
\int_0^1 \phi \circ T^k_{\gamma + \delta} \, dm 
  - \int_0^1 \phi \circ T^k_\gamma \, dm
= \sum_{j=0}^{k-1} \int_0^1
 \phi \cdot \cL_{\gamma+\delta}^j ( 
     \cL_{\gamma + \delta} - \cL_\gamma )
\cL_{\gamma}^{k-j-1} ( \mathbf{1} ) \, dm, 
\end{align}
we obtain the decomposition 
\begin{align}\label{eq:decomp_pk}
P_k = Q_k^{(1)} + Q_k^{(2)} + R_k^{(1)} + R_k^{(2)}, 
\end{align}
where
\begin{align*}
  &Q_k^{(i)} =
  \frac{ \sum_{j=0}^{k-1} \delta_i 
  \int_0^1 \phi \cdot 
  \cL_{\gamma + \delta}^j \partial_i \cL_\gamma ( \varphi_j ) \, dm - D_i \delta_i   }{
    \Vert \delta \Vert
  }, \\
  &R_k^{(i)}
  = \frac{
    \sum_{j=0}^{k-1} \int_{0}^1 \int_{0}^1
    \phi \cdot \cL_{\gamma + \delta}^j \partial_i^2 \cL_{\gamma + \delta t}(\varphi_j)
     \, dm \cdot (1 - t) \delta_i^2 \, dt  
  }{ \Vert \delta \Vert},
\end{align*}
and 
$
\varphi_j = \cL_{\gamma}^{k-j-1}( \mathbf{1} ).
$ 
The interchange of the order of integration 
in $R_k^{(i)}$ can be justified by applying 
Fubini's theorem together with 
Lemmas \ref{lem:bounds_on_second_1} 
and \ref{lem:bounds_on_second_2}.

\subsection{Limit of $R_k^{(i)}$} We 
show that if $\delta_*$ is 
sufficiently small, then 
\begin{align}\label{eq:estim_on_Rk}
\sup_{ \Vert \delta \Vert \le \delta_* , \Vert  \delta' \Vert \le \delta_* }\sum_{j=0}^\infty
\Vert \cL_{ \gamma + \delta }^j 
\partial_i^2 \cL_{\gamma + \delta' }(\varphi_j)
\Vert_{L^1(m)} < \infty, 
\end{align}
which implies that for some $C > 0$, 
$$
|R_k^{(i)}|
\le \frac{ C \Vert \phi \Vert_\infty \delta_i^2 }{ \Vert \delta \Vert} 
\le C \Vert \phi \Vert_\infty \Vert \delta \Vert \to 0,  
\quad \text{as $\delta \to 0$}.
$$
Suppose $i=1$. By Lemma 
\ref{lem:partial_second_L}, 
$m(\partial_1^2 \cL_{\gamma + \delta' 
}(\varphi_j)) = 0$. Moreover, since
$\varphi_j \in 
\cC_a(\gamma_u) \cap \cC^{(3)}_{b_1,b_2,b_3}$, it follows from Lemma
\ref{lem:bounds_on_second_1} that 
\begin{align}\label{eq:decomp_rk}
\partial_1^2 \cL_{\gamma + \delta' }(\varphi_j)
= (\psi_1^{(\delta')} - m(\psi_1^{(\delta')}) - (\psi_2^{(\delta')} - m(\psi_2^{(\delta')})) , 
\end{align}
where $\psi_i^{(\delta')} \in 
\cC_{a}(  \delta_* , \beta_u )$ 
and $\Vert \psi_i^{(\delta')} \Vert_{L^1(m)} 
\le C_i(\gamma, \delta_*, a, b_1, b_2, b_3) < \infty$ 
for all $\delta'$ 
with $\Vert \delta' \Vert \le \delta_*$. 
Since $\cC_{a}(\gamma_u)$ 
contains constant functions, 
Lemma \ref{lem:memory_loss} implies the 
upper bound 
$$
\Vert \cL_{ \gamma + \delta }^j 
\partial_1^2 \cL_{\gamma + \delta' }(\varphi_j)
\Vert_{L^1(m)}
\le \sum_{t=1,2} \Vert 
\cL_{\gamma + \delta}^j( 
  \psi_t^{(\delta')} - m(\psi_t^{(\delta')}
)) \Vert_{L^1(m)} 
\le C j^{-\gamma_*},
$$
where $C = C(\gamma, \delta_*, a, b_1, b_2,b_3) > 0$ is a 
constant and 
$$
\gamma_* =  \tfrac{1}{\alpha_u}(
        \tfrac{1}{\beta_u} - \delta_* ).
$$
The obtained upper bound is summable over 
$j$ if $\delta_* 
< \tfrac{1}{\beta_u} - \alpha_u$, which holds 
whenever $\delta_*$ is sufficiently small. Consequently, 
\eqref{eq:estim_on_Rk} holds for 
$i = 1$. For $i=2$ the proof is the same, 
except that Lemma \ref{lem:bounds_on_second_2}
is applied instead of Lemma 
\ref{lem:bounds_on_second_1} to obtain 
the decomposition of 
$\partial_2^2 \cL_{\gamma + \delta' }(\varphi_j)$.

\subsection{Limit of $Q_k^{(i)}$}
To complete the proof of Theorem \ref{thm:main}, 
it remains to show that 
\begin{align*}
  |Q_k^{(i)}| &= \biggl|
\frac{ \sum_{j=0}^{k-1} \delta_i \int_0^1 \phi \cdot 
  \cL_{\gamma + \delta}^j \partial_i \cL_\gamma
   ( \varphi_j ) \, dm - D_i \delta_i   }{
    \Vert \delta \Vert
  } \biggr| \\
  &\le \biggl| \sum_{j=0}^{k-1}  \int_0^1 \phi \cdot 
  \cL_{\gamma + \delta}^j \partial_i \cL_\gamma
   ( \varphi_j ) \, dm - D_i  \biggr| \to 0
\end{align*}
as $\delta \to 0$, where 
$k = k_0(\xi, \delta, B) 
= \lceil \mathfrak{c} 
\Vert \delta \Vert^{- (1 + \xi) 
\frac{\alpha_u\beta_u}{1 - \alpha_u\beta_u}} \rceil$. 

Let $\ve > 0$. 
Recalling $\partial_i \cL_\gamma(\varphi_j)
= (X_{\gamma,i} \cN_{\gamma,i}(\varphi_j))'$, 
by Lemma
\ref{lem:D_well_defined}
there exists $N_\ve > 1$ 
such that whenever 
$\Vert \delta \Vert$ is sufficiently small, 
i.e. $k$ is sufficiently large, 
\begin{align*}
    &\biggl| \sum_{j=0}^{k-1}  \int_0^1 \phi \cdot 
    \cL_{\gamma + \delta}^j \partial_i \cL_\gamma
     ( \varphi_j ) \, dm - D_i \biggr| \\
     &\le \biggl|
      \sum_{j=0}^{N_\ve}  \int_0^1 \phi \cdot 
     \cL_{\gamma + \delta}^j [
      (X_{\gamma,i} \cN_{\gamma,i}(\varphi_j))'
     ] \, dm - \sum_{j=0}^{N_\ve} 
      \int_0^1
      \phi \cdot \cL_{\gamma}^j[ ( X_{\gamma,i} 
      \cN_{\gamma,i}
      (h_{\gamma})  )'  ] \, dm \biggr|
      + \tfrac{\ve}{2} \\
    &= \biggl| \sum_{j=0}^{N_\ve} \int 
    \phi \cdot  ( \cL_{\gamma + \delta}^j 
    - \cL_\gamma^j  ) 
    [
      (X_{\gamma,i} \cN_{\gamma,i}(\varphi_j))'
     ] \, dm  \biggr|
      \\
    &+ \biggl| \sum_{j=0}^{N_\ve} 
    \int_0^1
    \phi \cdot \cL_{\gamma}^j[ 
        (X_{\gamma,i} \cN_{\gamma,i}(\varphi_j))'
        -   
    ( X_{\gamma,i} 
    \cN_{\gamma,i}
    (h_{\gamma})  )'  ] \, dm \biggr| + \tfrac{\ve}{2},
  \end{align*}
  where we recall $
  \varphi_j = \cL_{\gamma}^{k-j}( \mathbf{1} )
  $. Hence, it suffices to show that whenever 
  $k$ is sufficiently large, 
  \begin{align}\label{eq:to_show_last_1}
    \Vert ( \cL_{\gamma + \delta}^j 
    - \cL_\gamma^j  ) 
    [
      (X_{\gamma,i} \cN_{\gamma,i}(\varphi_j))'
     ]
    \Vert_{L^1(m)}
    \le \widetilde{\ve}   
    \quad  \forall  \, 0 \le j \le N_\ve,
  \end{align}
  and 
  \begin{align}\label{eq:to_show_last_2}
    \Vert
    (X_{\gamma,i} \cN_{\gamma,i}(\varphi_j))' -   
    ( X_{\gamma,i} 
    \cN_{\gamma,i}
    (h_{\gamma})  )' 
    \Vert_{L^1(m)} \le \widetilde{\ve}    
    \quad  \forall  \, 0 \le j \le N_\ve,
  \end{align}
  where 
  $
  \widetilde{\ve} = 
  \tfrac{\varepsilon}{4 N_\ve (1 + \Vert 
    \phi \Vert_\infty) }.
  $

\subsubsection{$i=1$} 
For \eqref{eq:to_show_last_1}, we use  
Lemma 
\ref{lem:D_well_defined} to decompose 
$$
(X_{\gamma,1} \cN_{\gamma,1}(\varphi_j))'
= \psi_1^{(j)} - \psi_2^{(j)}
$$
for some $\psi_r^{(j)} \in \cC_a^{(2)}(\gamma)$. 
%Since $\cL_\gamma$ is an $L^1$-contraction, 
Hence, 
\begin{align*}
  &\Vert ( \cL_{\gamma + \delta}^j 
  - \cL_\gamma^j  ) 
  [
    (X_{\gamma,1} \cN_{\gamma,1}(\varphi_j))'
   ] \Vert_1 \\ 
  &\le \sum_{\ell=0}^{j} \sum_{r=1,2}
  \Vert 
  ( \cL_{\gamma + \delta}
   - \cL_{\gamma} ) \cL_\gamma^{\ell} \psi_r^{(j)} 
   \Vert_{L^1(m)} \\
  &= \sum_{\ell=0}^{j} \sum_{r=1,2}
  \biggl\Vert \delta_1 
  \int_0^1 \partial_1 \cL_{\gamma + \delta t } 
  \cL_\gamma^{\ell} \psi_r^{(j)}  dt  
  + 
  \delta_2 
  \int_0^1 \partial_2 \cL_{\gamma + \delta t } 
  \cL_\gamma^{\ell} \psi_r^{(j)}  dt 
   \biggr\Vert_{L^1(m)} \\
  &\le 
  \sum_{\ell=0}^{j} \sum_{r=1,2} \biggl[
  | \delta_1 |
  \int_0^1 \Vert \partial_1 \cL_{\gamma + \delta t } 
  \cL_\gamma^{\ell} \psi_r^{(j)} \Vert_{L^1(m)}  dt \\
  &+ 
  | \delta_2 |
  \int \Vert \partial_2 \cL_{\gamma + \delta t } 
  \cL_\gamma^{\ell} \psi_r^{(j)} \Vert_{L^1(m)}  dt 
  \biggr],
\end{align*}
where the  $L^1$ contraction property 
of $\cL_\gamma$ was used in the first inequality.
Now $
\cL_\gamma^{\ell} \psi_r^{(j)} \in 
\cC_a^{(2)}(\gamma)$ so that, 
by Lemmas \ref{lem:x_alpha} and \ref{lem:first_partial_L}
$$
\Vert \partial_i \cL_{\gamma + \delta t } 
  \cL_\gamma^{\ell} \psi_r^{(j)} \Vert_{L^1(m)} 
  \le C(\gamma, \delta_*,  a) < \infty.
$$
This establishes \eqref{eq:to_show_last_1}.

To show \eqref{eq:to_show_last_2}, we first 
use Lemmas \ref{lem:x_alpha} and 
\ref{lem:memory_loss}  to derive the 
upper bound 
\begin{align*}
  \Vert X_{\gamma, 1}' (
    \cN_{\gamma, 1}(\varphi_j) 
    - \cN_{\gamma, 1}(h_\gamma)
   )
  \Vert_{L^1(m)}
  &\le 
  C(\alpha) \Vert \cN_{\gamma, 1}(\varphi_j) 
  - \cN_{\gamma, 1}(h_\gamma) \Vert_{L^1(m)} \\
  &\le C(\alpha) \Vert \varphi_j
  - h_\gamma \Vert_{L^1(m)}
   \le C(\gamma)  (k - N_\ve )^{ 
    1- \frac{1}{\alpha \beta }  },
\end{align*}
whenever $k > N_\ve$. Hence,
\begin{align*}
  &\Vert
  (X_{\gamma, 1} \cN_{\gamma, 1}(\varphi_j))' -   
  ( X_{\gamma, 1} 
  \cN_{\gamma, 1}
  (h_{\gamma})  )' 
  \Vert_{L^1(m)} \\ 
  &\le \Vert X_{\gamma, 1} 
  [\cN_{\gamma, 1} (\varphi_j)]' 
  -  X_{\gamma, 1} 
  [\cN_{\gamma, 1} (h_\gamma)]' 
  \Vert_{L^1(m)} 
  + C(\gamma)  (k - N_\ve )^{  
   1- \frac{1}{\alpha \beta }  }.
\end{align*}
The remaining term will be controlled using 
Lemma \ref{lem:distortion_returns} and
\eqref{eq:dist_sharp}.

Let $\ell > 1$ be an integer 
whose value will be fixed later. 
We decompose 
\begin{align}
  &\Vert X_{\gamma, 1} 
  [\cN_{\gamma, 1} (\varphi_j)]' 
  -  X_{\gamma, 1} 
  [\cN_{\gamma, 1} (h_\gamma)]' 
  \Vert_{L^1(m)} \notag \\
  &\le \int_0^{b_\ell} 
  | X_{\gamma, 1} 
  [\cN_{\gamma, 1} (\varphi_j)]' 
  -  X_{\gamma, 1} 
  [\cN_{\gamma, 1} (h_\gamma)]' | \, d
  m \label{eq:control_last_1}\\
  &+ 
  \int_{b_\ell}^{1} 
  | X_{\gamma, 1} 
  [\cN_{\gamma, 1} (\varphi_j)]' 
  -  X_{\gamma, 1} 
  [\cN_{\gamma, 1} (h_\gamma)]' | \, d
  m . \label{eq:control_last_2}
\end{align}
By Lemmas \ref{lem:x_alpha} 
and \ref{lem:cone_in_c3}, for 
$\varphi \in \{ \varphi_j, h_\gamma \}$,
we have
\begin{align*}
  \eqref{eq:control_last_1} &\le 
  C(\gamma, a, b_1)  \int_0^{b_\ell} 
  x^{1+\alpha}( 1 - \log ( x) ) \cdot 
  x^{-1} x^{ \frac{1}{\beta} - \alpha - 1 } \, dx \\ 
  &\le 
  C(\gamma, a, b_1) \int_0^{b_\ell} 
  ( 1 - \log ( x) ) \cdot 
   x^{ \frac{1}{\beta} - 1 } \, dx 
   \le C(\gamma, a, b_1) b_\ell^{ \frac{1}{\beta} } 
   (1 - \log ( b_\ell )) 
   \le C(\gamma, a, b_1) \ell^{ - \frac{1}{2\alpha\beta} }.
\end{align*}
Using $|X_{\gamma, 1}| \le C(\alpha) < \infty$, 
a change 
of variable, and Lemma \ref{lem:memory_loss}, 
we see that 
\begin{align*}
\eqref{eq:control_last_2} &\le 
C(\alpha) \int_{b_\ell}^{1} 
  |  
  [\cN_{\gamma, 1} (\varphi_j - h_\gamma)]' 
 | \, d
  m 
  \le C(\gamma) \ell^{ \frac{1}{\alpha} - 1 } 
  \Vert \varphi_j - h_\gamma  
  \Vert_{L^1(m)} + 
  C(\gamma) \int_{b_{\ell + 1}}^1 | 
   \varphi_j' - h_\gamma'  | \, dm \\ 
   & \le C(\gamma, \ell) 
   \biggl(
    (k - N_\ve )^{  
    1- \frac{1}{\alpha \beta }  } 
   + \int_{b_{\ell + 1}}^1 | 
   \varphi_j' - h_\gamma'  | \, dm \biggr).
\end{align*}
It remains to control 
\begin{align}\label{eq:integral_i}
I := \int_{b_{\ell + 1}}^1 | 
   \varphi_j' - h_\gamma'  | \, dm = 
   \int_{b_{\ell + 1}}^1 | 
     [\cL^m_\gamma(u)]'  | \, dm,
\end{align}
where $u = 
\cL_\gamma^{k - j - m}( \mathbf{1} - h_\gamma  )$ 
and $m = \round{\tfrac{k}{2}}$. 
We have 
$|
[ \cL^m_\gamma (u ) ]'(x) | 
\le E_1(x) + E_2(x)$, where 
\begin{align*}
   E_1(x) = \sum_{y \in T^{-m}_\gamma x } \frac{
    \frac{1}{(T^m_\gamma)' 
    y } | u'(y) | 
    }{(T^m_\gamma)' y }
= \cL_\gamma^m \biggl(  \tfrac{|u'|}{(T^m)'}   \biggr)(x) 
\quad \text{and} \quad 
E_2(x) = 
    \sum_{y \in T^{-m}_\gamma x } \frac{
      \frac{ (T^m_\gamma)''y }{((T^m_\gamma)' 
      y )^2} | u(y) | 
      }{(T^m_\gamma)' y }.
\end{align*}

\begin{claim}\label{claim:control_e1} There 
    exists $C = C(\gamma, \ell, a, b_1) > 0$ 
    such that 
\begin{align*}
    \int_{b_{\ell+1}}^1 E_1(x) \, dx 
    \le C k^{ - \frac{1}{(\alpha+1)\beta} }.
\end{align*}
\end{claim}

\begin{proof}[Proof of Claim \ref{claim:control_e1}]
For brevity, denote $T = T_\gamma$. We have 
\begin{align*}
\int_{b_{\ell+1}}^1 E_1(x) \, dx 
&= \int_{ T^{-m}[b_{\ell+1}, 1] } 
\frac{|u'(x)|}{(T^m)'x} \, dx \\ 
&= \sum_{j=0}^\infty 
\int_{ T^{-m}[b_{\ell+1}, 1] 
\cap [b_{j+1}, b_j]  } \frac{|u'(x)|}{(T^m)'x} 
\, dx + 
\int_{ T^{-m}[b_{\ell+1}, 1] \cap [\frac12, 1] }
\frac{|u'(x)|}{(T^m)'x} 
\, dx.
\end{align*}
Since $u$ is the difference of two 
densities in $\cC_a(\gamma) \cap 
\cC^{(3)}_{b_1,b_2,b_3}$, 
$|u'|$ is bounded on $[\tfrac12,1]$.
Hence, an application 
of \eqref{eq:dist_first} yields the upper bound 
$$
\int_{ T^{-m}[b_{\ell+1}, 1] \cap [\frac12, 1] }
\frac{|u'(x)|}{(T^m)'x} 
\, dx \le C(a, b_1, \gamma, \ell) 
m^{-1-\frac{1}{\alpha\beta}}.
$$

To control the remaining integral, we let 
$1 < p \ll m$ be an integer whose value will be fixed 
later, and decompose 
\begin{align}
&\sum_{j=0}^\infty 
\int_{ T^{-m}[b_{\ell+1}, 1] 
\cap [b_{j+1}, b_j]  } \frac{|u'(x)|}{(T^m)'x} 
\, dx \notag \\ 
&= \sum_{j=p}^\infty 
\int_{ T^{-m}[b_{\ell+1}, 1] 
\cap [b_{j+1}, b_j]  } \frac{|u'(x)|}{(T^m)'x} 
\, dx 
+ 
\int_{ T^{-m}[b_{\ell+1}, 1] 
\cap [b_{p}, \frac12]  } \frac{|u'(x)|}{(T^m)'x} 
\, dx. \label{eq:integrals_e1}
\end{align}
The second integral in 
\eqref{eq:integrals_e1} is estimated using 
properties 
of densities in $\cC_a(\gamma) \cap 
\cC^{(3)}_{b_1,b_2,b_3}$ together with 
\eqref{eq:dist_first} and Lemma 
\ref{lem:bn_bound}: for some 
constant $C = C(\gamma, \ell, a, b_1) > 0$,
\begin{align*}
    &\int_{ T^{-m}[b_{\ell+1}, 1] 
    \cap [b_{p}, \frac12]  } \frac{|u'(x)|}{(T^m)'x} 
    \, dx 
    \le C   
    m^{-1-\frac{1}{\alpha\beta}} 
    \int_{b_p}^{\frac12} x^{ \frac{1}{\beta} 
    - \alpha - 2 } \, dx \\ 
    &\le 
    C  
    m^{-1-\frac{1}{\alpha\beta}} 
    b_{p}^{ \frac{1}{\beta} 
    - \alpha - 1 } 
    \le C m^{-1-\frac{1}{\alpha\beta}} 
    p^{ 1 + \frac{1}{\alpha} - 
    \frac{1}{\alpha\beta}  }.
\end{align*}

For the first integral in \eqref{eq:integrals_e1}, 
note that if $x \in T^{-m}[b_{\ell+1}, 1] 
\cap [b_{j+1}, b_j]$, then starting 
from $x$ it will take $j+1$ 
iterates until the trajectory first enters 
$[\tfrac12, 1]$, and so in 
\eqref{eq:dist_sharp} we have $\ell_1 = j + 1$ 
if $j < m$.
It follows that, for some constant 
$C = C(\ell, \gamma) > 0$, 
$$
\frac{1}{(T^m)'x} \le 
C \min\{j, m\}^{-1-\frac{1}{\alpha}}
$$
holds whenever $j \ge 0$ and 
$x \in 
T^{-m}[b_{\ell+1}, 1] 
\cap [b_{j+1}, b_j]$. Then, 
using properties of densities 
in 
$\cC_a(\gamma) \cap 
\cC^{(3)}_{b_1,b_2,b_3}$
together with Lemma \ref{lem:bn_bound},
we see that 
\begin{align*}
    &\sum_{j=p}^\infty 
    \int_{ T^{-m}[b_{\ell+1}, 1] 
    \cap [b_{j+1}, b_j]  } \frac{|u'(x)|}{(T^m)'x} 
    \, dx \\ 
    &= 
    \sum_{j=p}^{m-1} 
    \int_{ T^{-m}[b_{\ell+1}, 1] 
    \cap [b_{j+1}, b_j]  } \frac{|u'(x)|}{(T^m)'x} 
    \, dx + \int_{ T^{-m}[b_{\ell+1}, 1] 
    \cap [0, b_m]  } \frac{|u'(x)|}{(T^m)'x} 
    \, dx \\ 
    &\le 
    \sum_{j=p}^{ \infty }C j^{-1 - \frac{1}{\alpha}}
    \int_{ [b_{j+1}, b_j]  } |u'(x)| 
    \, dx +  C m^{-1-\frac{1}{\alpha}} 
    \int_{ [b_{\ell+m+1}, 1]  } |u'(x)| 
    \, dx  \\ 
    &\le 
    C_1 \sum_{j=p}^\infty j^{-1 - \frac{1}{\alpha}}
    | [b_{j+1}, b_j] | b_{j+1}^{ \frac{1}{\beta} 
    - \alpha - 2  } 
    + C_2  m^{-1-\frac{1}{\alpha}} b_{\ell+m+1}^{
        \frac{1}{\beta} - \alpha - 1
    } \\ 
    &\le C_1 \sum_{j=p}^\infty j^{-2 - \frac{2}{\alpha}}
    j^{ \frac{2}{\alpha} + 1 - \frac{1}{\alpha\beta}  }
    + C_2  m^{-1-\frac{1}{\alpha}} 
    m^{ \frac{1}{\alpha} + 1 - \frac{1}{\alpha\beta} } \\ 
    &\le 
    C_1 p^{-\frac{1}{\alpha\beta}} + C_2
    m^{-\frac{1}{\alpha\beta}},
\end{align*}
where $C_i = C_i(\gamma, \ell, a, b_1) > 0$.

We have established 
$$
\int_{b_{\ell+1}}^1 E_1(x) \, dx 
    \le C \biggl( m^{-1-\frac{1}{\alpha\beta}} 
    + m^{-1-\frac{1}{\alpha\beta}} 
    p^{ 1 + \frac{1}{\alpha} - 
    \frac{1}{\alpha\beta}  } + 
    p^{ - \frac{1}{\alpha\beta} } + 
    m^{-\frac{1}{\alpha\beta}} \biggr).
$$
The desired upper bound follows by 
choosing $p = O( m^{\frac{\alpha\beta}{(\alpha+1)
\beta}} )$, and by recalling that 
$m = \round{\tfrac{k}{2}}$.
\end{proof}

\begin{claim}\label{claim:control_e2} There 
    exists $C = C(\gamma, \ell, a) > 0$ 
    such that 
\begin{align*}
    \int_{b_{\ell+1}}^1 E_2(x) \, dx 
    \le C (k - 2N_\ve )^{ 1 - \frac{1}{\alpha\beta} }.
\end{align*}
\end{claim}

\begin{proof}[Proof of Claim 
    \ref{claim:control_e2}] By 
    \eqref{eq:dist_second}, 
    \begin{align*}
        \int_{b_{\ell+1}}^1 E_2(x) \, dx 
        &=  \int_{b_{\ell+1}}^1 \sum_{y \in T^{-m}_\gamma x } \frac{
            \frac{ (T^m_\gamma)''y }{((T^m_\gamma)' 
            y )^2} | u(y) | 
            }{(T^m_\gamma)' y } \, dx  \\ 
        &\le C \int_{b_{\ell+1}}^1 \sum_{y \in T^{-m}_\gamma x } \frac{
             | u(y) | 
            }{(T^m_\gamma)' y } \, dx 
        = C \int_{0}^1 \cL_{\gamma}^m(|u|) \, dm
        \le C \Vert u \Vert_{L^1(m)}, 
    \end{align*}
where the $L^1$ contraction property of 
$\cL_\gamma$ was used in the last inequality.
Recalling $u = 
\cL_\gamma^{k - j - m}( \mathbf{1} - h_\gamma  )$ 
and $m = \round{\tfrac{k}{2}}$, it follows 
by an  application of 
Lemma \ref{lem:memory_loss} that 
$$
\int_{b_{\ell+1}}^1 E_2(x) \, dx  
\le C (k - 2j )^{ 1 - \frac{1}{\alpha\beta} }
\le C (k - 2N_\ve )^{ 1 - \frac{1}{\alpha\beta} }.
$$
\end{proof}

By Claims \ref{claim:control_e1} and 
\ref{claim:control_e2}, 
$$
I \le C  \biggl( k^{ - \frac{1}{(\alpha+1)\beta} } 
+ (k - 2N_\ve )^{ 1 - \frac{1}{\alpha\beta} }\biggr),
$$
and so it follows that 
$$
\eqref{eq:control_last_2} 
\le 
 C  \biggl( k^{ - \frac{1}{(\alpha+1)\beta} } 
+ (k - 2N_\ve )^{ 1 - \frac{1}{\alpha\beta} }\biggr),
$$
where $C = C(\gamma, \ell, a, b_1) > 0$.
Combining this with the estimate 
on \eqref{eq:control_last_1} obtained earlier, we 
see that 
\begin{align*}
&\Vert
(X_{\gamma,1} \cN_{\gamma,1}(\varphi_j))' -   
  ( X_{\gamma, 1} 
  \cN_{\gamma,1}
  (h_{\gamma})  )' 
  \Vert_{L^1(m)} \\
  &\le C_1(\gamma, a) \ell^{ - \frac{1}{2\alpha\beta}} + 
  C_2(\gamma, \ell, a, b_1)  \biggl( k^{ - \frac{1}{(\alpha+1)\beta} } 
  + (k - 2N_\ve )^{ 1 - \frac{1}{\alpha\beta} }\biggr).
\end{align*}
Now, fix $\ell$ sufficiently large such that the first 
term is smaller than 
$\widetilde{\ve}/2$. Then for all sufficiently 
large $k$ (depending on $\ell$) the second  
term is smaller than 
$\widetilde{\ve}/2$. We have established
\eqref{eq:to_show_last_2}.

%%%%%%%%%%%%%%%%%%%%%%%%%%%%%%%%%%%%%%
%%%%%%%%%%%%%%%%%%%%%%%%%%%%%%%%%%%%%%
%%%%%%%%%%%%%%%%%%%%%%%%%%%%%%%%%%%%%%

\subsubsection{$i=2$}
We obtain \eqref{eq:to_show_last_1} as in 
the case $i = 1$, using Lemma 
\ref{lem:D_well_defined}. 
For \eqref{eq:to_show_last_2}, we first 
apply the estimate on 
$X_{\gamma,2}'$ from Lemma \ref{lem:x_beta} together 
with $|\cN_{\gamma, 2}(\varphi)(x)| \le 
C(\gamma) x^{ \frac{1}{\beta} - 1 }$ for 
$\varphi \in \{ \varphi_j, h_\gamma \}$ 
to obtain  
\begin{align*}
  &\Vert X_{\gamma, 2}' [
    \cN_{\gamma, 2}(\varphi_j) 
    - \cN_{\gamma, 2}(h_\gamma)
  ]
  \Vert_{L^1(m)} \\
  &\le 
  \int^{x_*}_0 | X_{\gamma, 2}' [
    \cN_{\gamma, 2}(\varphi_j) 
    - \cN_{\gamma, 2}(h_\gamma)
  ] | \, dx \\ 
  &+ \int_{x_*}^1 | X_{\gamma, 2}' [
    \cN_{\gamma, 2}(\varphi_j) 
    - \cN_{\gamma, 2}(h_\gamma)
  ] | \, dx \\
  & \le C(\gamma) \log(x_*^{-1}) x^{ \frac1\beta  }_*
  + 
  C(\gamma) \log( x_*^{-1} ) \Vert 
  \varphi_j
    - h_\gamma
  \Vert_{L^1(m)} 
\end{align*}
for any $x_* \in (0,1)$. Then, an application of 
Lemma \ref{lem:memory_loss} together 
with a suitable choice of $x_*$ yields
$$
\Vert X_{\gamma, 2}' [
    \cN_{\gamma, 2}(\varphi_j) 
    - \cN_{\gamma, 2}(h_\gamma)
  ]
  \Vert_{L^1(m)} \le 
  C(\gamma) (k - N_\ve )^{ \frac{1}{2} 
  ( 1- \frac{1}{\alpha \beta } ) }, 
$$
from which we see that 
\begin{align*}
  \eqref{eq:to_show_last_2} \le \Vert X_{\gamma, 2} 
  [\cN_{\gamma, 2} (\varphi_j)]' 
  -  X_{\gamma, 2} 
  [\cN_{\gamma, 2} (h_\gamma)]' 
  \Vert_{L^1(m)} + C(\gamma) (k - N_\ve )^{
     \frac{1}{2} 
  ( 1- \frac{1}{\alpha\beta } ) }.
\end{align*}

For the remaining term, as in the case 
$i = 1$, we 
let $\ell > 1$ be a number whose value is 
fixed later and decompose 
\begin{align}
  &\Vert X_{\gamma, 2} 
  [\cN_{\gamma, 2} (\varphi_j)]' 
  -  X_{\gamma, 2} 
  [\cN_{\gamma, 2} (h_\gamma)]' 
  \Vert_{L^1(m)} \notag \\
  &\le \int_0^{b_\ell} 
  | X_{\gamma, 2} 
  [\cN_{\gamma, 2} (\varphi_j)]' 
  -  X_{\gamma, 2} 
  [\cN_{\gamma, 2} (h_\gamma)]' | \, d
  m \label{eq:control_last_11}\\
  &+ 
  \int_{b_\ell}^{1} 
  | X_{\gamma, 2} 
  [\cN_{\gamma, 2} (\varphi_j)]' 
  -  X_{\gamma, 2} 
  [\cN_{\gamma, 2} (h_\gamma)]' | \, d
  m . \label{eq:control_last_22}
\end{align}
By Lemmas \ref{lem:x_beta} and 
\ref{lem:cone_in_c3} combined with 
$|\cN_{\gamma, 2}(\varphi)(x)| \le 
C x^{ \frac{1}{\beta} - 1 }$ for 
$\varphi \in \{ \varphi_j, h_\gamma \}$, 
we have 
\begin{align*}
  \eqref{eq:control_last_11} &\le 
  C \int_0^{b_\ell} 
  x  \log ( x^{-1})  \cdot 
  x^{-1} x^{ \frac{1}{\beta} - 1 } \, dx
  \le 
  C \int_0^{b_\ell} 
  \log ( x^{-1} )  \cdot 
   x^{ \frac{1}{\beta} - 1 } \, dx \\
   &\le C \log( b_\ell^{-1} ) 
   b_\ell^{ \frac{1}{\beta} }
   \le C \ell^{ - 
   \frac{1}{2}  \frac{1}{\alpha\beta}   },
\end{align*}
for some $C = C(a, b_1, \gamma) > 0$ 
independent of $\ell$.

Using $|X_{\gamma, 2}| \le C(\gamma)$ together 
with a change of variable, we obtain 
\begin{align*}
\eqref{eq:control_last_22} 
&\le 
C(\gamma)  \int_{b_\ell}^{1} 
  |  
  [\cN_{\gamma, 2} (\varphi_j - h_\gamma)]' 
 | \, dm \\
  &\le C(\gamma) \ell^{ \frac{1}{\alpha}( 2 - \frac{1}{\beta} ) } \Vert \varphi_j - h_\gamma 
  \Vert_{L^1(m)} + C(\gamma )
  \ell^{ \frac{1}{\alpha} (  1 - \frac{1}{\beta} ) }
  \int_{ \hat{b}_{\ell + 1} }^1 |  \varphi_j' - h_\gamma' | \, dm \\
  &\le C(\gamma) \ell^{ \frac{1}{\alpha}( 2 - \frac{1}{\beta} ) } \Vert \varphi_j - h_\gamma 
  \Vert_{L^1(m)} + C(\gamma )
  \ell^{ \frac{1}{\alpha} (  1 - \frac{1}{\beta} ) }
  \int_{ b_{\ell + 1} }^1 |  \varphi_j' - h_\gamma' | \, dm,
\end{align*}
and now it follows by applying Lemma 
\ref{lem:memory_loss} that 
\begin{align*}
\eqref{eq:control_last_22} &\le 
C(\gamma) \ell^{ \frac{1}{\alpha}( 2 
- \frac{1}{\beta} ) } 
(k - N_\ve )^{  
    1- \frac{1}{\alpha\beta}  } + C(\gamma )
  \ell^{ \frac{1}{\alpha} (  1 - \frac{1}{\beta} ) }
  \int_{ b_{\ell + 1} }^1 |  \varphi_j' - h_\gamma' | \, dm.
\end{align*}
The remaining integral is exactly the 
same as \eqref{eq:integral_i},
and thus to complete the proof of 
\eqref{eq:to_show_last_2} we may 
repeat the argument for $i=1$.

%%%%%%%%%%%%%%%%%%%%%%%%%%%%
%%%%%%%%%%%%%%%%%%%%%%%%%%%
%%%%%%%   APPENDIX 1 %%%%%%

\appendix

\section{Proof of
 Lemma \ref{lem:cone_in_c3}}\label{appendix:a}

Since $ \cL_\gamma(\varphi) = \sum_{i=1,2}
\cN_{\gamma,i}(\varphi)$, it 
suffices to verify the statement concerning
$\cN_{\gamma,2}$.

\begin{lem}\label{lem:inv_na} There 
    exist (absolute) constants $\kappa_i > 1$
    such that, if 
    $b_1 \ge 1$, $b_2 \ge \kappa_1 b_1$, 
    and $b_3  \ge \kappa_2 b_2$, then 
    for all $\beta \ge 1$, 
  $
  \cN_{\gamma, 2} (
  \cC^{(3)}_{b_1,b_2,b_3} )
  \subset \cC^{(3)}_{b_1,b_2,b_3}.
  $
\end{lem}

\begin{proof} Let $\varphi \in \cC^{(3)}_{b_1,b_2,b_3}$. We identify conditions on 
  $b_1, b_2$ and $b_3$ that guarantee 
  $\cN_{\gamma, 2}(\varphi) \in \cC^{(3)}_{b_1,b_2,b_3}$. 
  Let $x \in (0,1)$. Throughout this 
  proof we denote 
  $y = g_{\gamma, 2}(x) \in [\tfrac12, 1]$ 
  and $T = T_\gamma$.
  
\noindent\textbf{Step 1:} $|\cN_{\gamma, 2}(\varphi)'(x)| \le \tfrac{b_1}{x} \varphi(x)$. We have
\begin{align*}
  | \cN_{\gamma, 2}(\varphi)'(x) | 
  &= \biggl| -\frac{ T_\gamma''(y) }{ (T'_\gamma(y))^3 } \varphi(y) + \frac{1}{(T'_\gamma(y))^2} \varphi'(y)
  \biggr| \\
  &\le \frac{\varphi(y) }{T_\gamma'(y)} \biggl\{ 
  \frac{T''_\gamma(y)}{(T_\gamma'(y))^2} + \frac{b_1}{y T_\gamma'(y)}  
  \biggr\} \\
  &= \frac{b_1}{x} \cN_{\beta, 2}(\varphi)(x) 
  \cdot \biggl\{ 
    \frac{T_\gamma(y)}{b_1} \biggl( \frac{T''_\gamma(y)}{(T_\gamma'(y))^2} 
    + \frac{b_1}{y T_\gamma'(y)}
    \biggr)
  \biggr\},
\end{align*}
where 
$\varphi \in \cC^{(3)}_{b_1,b_2,b_3}$ was used in the second inequality.
The term inside curly brackets is equal to 
\begin{align*}
  &\frac{2^\beta ( y - \tfrac12 )^{\beta}}{b_1} 
  \biggl[ 
    \frac{\beta - 1}{\beta 2^\beta ( y - \tfrac12)^\beta} + \frac{b_1}{y\beta 2^{\beta}
    (y- \tfrac12)^{\beta - 1}}
  \biggr] \\ 
  &= \frac{1}{b_1} (  1 - \frac{1}{\beta} ) + \frac{1}{\beta} - \frac{1}{2y\beta} \le 1,
\end{align*}
whenever $b_1 \ge 1$.

\noindent\textbf{Step 2:} $|\cN_{\gamma, 2}(\varphi)''(x)| \le \tfrac{b_2}{x^2} \varphi(x)$.
Using 
$\varphi \in \cC^{(3)}_{b_1,b_2,b_3}$, we see that 
\begin{align*}
  &|\cN_{\gamma, 2}(\varphi)''(x)| \\
  &\le \frac{b_2}{x^2} \cN_{\gamma, 2}(\varphi)(x)  
  \biggl\{ 
    \frac{(T(y))^2}{y^2} \frac{1}{ (T'(y))^2 } + 3 \frac{(T(y))^2}{b_2} \frac{b_1}{y} \frac{T''(y)}{(T'(y))^3} \\
    &+ \frac{T(y)^2}{b_2} \frac{|T'''(y)|}{(T'(y))^3} 
    + 3\frac{T(y)^2}{b_2} \frac{T''(y)^2}{(T'(y))^4}
  \biggr\}.
\end{align*}
The term inside curly brackets is equal to 
\begin{align*}
  &\frac{1}{\beta^2} \biggl( 1 - \frac{1}{2y} \biggr)^2 + 3 \frac{b_1}{b_2} \frac{y - \tfrac12}{y} 
  \frac{\beta - 1}{\beta^2} 
   + \frac{|\beta -2 |(\beta - 1)}{\beta^2 b_2} + \frac{3}{b_2} \biggl( 1 - \frac{1}{\beta} \biggr)^2 \\ 
  &\le \frac14 + \frac{3}{2} \frac{b_1}{b_2} + \frac{1}{b_2} + \frac{3}{b_2} \le 1,
\end{align*}
if $b_2 \ge 12b_1 \ge 1$.

\noindent\textbf{Step 3:} 
$|\cN_{\gamma,2}(\varphi)'''(x)| \le \tfrac{b_2}{x^3} \varphi(x)$. Again by using 
$\varphi \in \cC^{(3)}_{b_1,b_2,b_3}$, we see that 
\begin{align*}
    |\cN_{\gamma, 2}(\varphi)'''(x)| 
    &\le 
    \varphi(y) \biggl[ 
        10 \frac{T''(y)|T'''(y)|}{(T'(y))^6}
        + 5 \frac{(T''(y))^3}{(T'(y))^7}
        + \frac{|T^{(4)}(y)|}{(T'(y))^5}
        \biggr]   \\ 
&+ |\varphi'(y)| \biggl[
    15 \frac{(T''(y))^2}{(T'(y))^6}
    + 4 \frac{|T'''(y)|}{(T'(y))^5}
 \biggr] \\
 &+ 6 |\varphi''(y)| \frac{T''(y)}{(T'(y))^5}   
 + |\varphi'''(y)| \frac{1}{(T'(y))^4} \\ 
 &\le 
 \frac{b_3}{x^3} \cN_{\gamma, 2}(\varphi)(x) 
 \cdot
 \frac{10 (T(y))^3 }{b_3}\biggl[ 
     \frac{T''(y)|T'''(y)|}{(T'(y))^5}
        +  \frac{(T''(y))^3}{(T'(y))^6}
        + \frac{|T^{(4)}(y)|}{(T'(y))^4}
    \biggr] \\ 
&+ \frac{b_3}{x^3}
\cN_{\gamma, 2}(\varphi)(x) \cdot \frac{15 b_1
(T(y))^3}{b_3 y}
\biggl[
    \frac{(T''(y))^2}{(T'(y))^5}
    +  \frac{|T'''(y)|}{(T'(y))^4}
 \biggr] \\ 
 &+ \frac{b_3}{x^3}
 \cN_{\gamma, 2}(\varphi)(x) \cdot 
 \frac{6 b_2 (T(y))^3 }{b_3 y^2} 
 \frac{T''(y)}{(T'(y))^4} 
 + \frac{b_3}{x^3} 
 \cN_{\gamma, 2}(\varphi)(x) \cdot
 \frac{(T(y))^3}{y^3} \frac{1}{(T'(y))^3}  \\ 
 &\le \frac{b_3}{x^3}
 \cN_{\gamma, 2}(\varphi)(x) \biggl\{ 
    \frac{(T(y))^3}{y^3 (T'(y))^3}
    + \frac{15 b_2}{b_3} \biggl[ 
        \frac{ (T(y))^3 }{ y^2} 
 \frac{T''(y)}{(T'(y))^4} 
 \\ 
 &+ \frac{
 (T(y))^3}{ y}
 \biggl(
     \frac{(T''(y))^2}{(T'(y))^5}
     +  \frac{|T'''(y)|}{(T'(y))^4}
  \biggr) \\ 
  &+  (T(y))^3 \biggl( 
    \frac{T''(y)|T'''(y)|}{(T'(y))^5}
       +  \frac{(T''(y))^3}{(T'(y))^6}
       + \frac{|T^{(4)}(y)|}{(T'(y))^4}
   \biggr)
    \biggr]
    \biggr\}
\end{align*}
Denote $\xi = 1 - \tfrac{1}{2y} \le \tfrac12$. 
The term inside curly brackets is equal to 
\begin{align*}
    &\frac{1}{\beta}\xi^{3} + 
    \frac{15b_2}{b_3} \biggl[  
        \frac{\beta - 1}{\beta^3} \xi^2
        + \frac{(\beta - 1)^2}{\beta^3} \xi
        + \frac{|\beta - 2|(\beta - 1)}{\beta^3} 
        \xi + \\
        &\frac{|\beta - 1|(\beta-1)^2}{\beta^4}
        + \frac{(\beta-1)^3}{\beta^3} 
        + \frac{|\beta - 3||\beta - 2| (\beta - 1)}{
            \beta^3
        }
        \biggr] \\ 
        &\le 
        \frac{1}{8} + \frac{90 b_2}{b_3} \le 1, 
\end{align*}
whenever $b_3 \ge 103 b_2$.

\end{proof}

\section{Proof of Theorem 
\ref{thm:main} for $\phi \in L^q$}\label{appendix:b}

For $\phi \in L^q$ with  $q $ sufficiently large, the linear response formula of 
Theorem \ref{thm:main} 
can be shown by modifying the proof for 
$\phi \in L^\infty(m)$ using a trick 
which involves 
a change of measure and interpolation.

\subsection{Loss of memory in $L^p$} In this section we assume that 
$\gamma = (\alpha, \beta) \in \fD$, and that 
$B = [\alpha_\ell, \alpha_u] \times [1, \beta_u] \subset \fD$, where 
$\gamma_u = (\alpha_u, \beta_u) = \gamma + (\delta_*, \delta_*)$  
and $\alpha_\ell = \alpha - \delta_*$
for a small $\delta_* > 0$.

Let $\widetilde{\cL}_\gamma$ denote 
the transfer operator associated to 
$T_\gamma$ and $\widetilde{m}$, where 
$\widetilde{m}$ denotes the 
measure 
defined by 
$
\widetilde{m}(A) = \int_{A} g \, dm
$
with $g(x) = x^{ \frac{1}{\beta} - 
\alpha - 1 }$. For $\varphi \in L^1(\widetilde{m})$
we have
$$
\widetilde{\cL}_\gamma^n(\varphi) = 
g^{-1} \cL_\gamma^n ( 
    g \varphi
) \quad \forall n \ge 1.
$$

The following extension of Lemma \ref{lem:memory_loss}
is derived using similar techniques to those found in \cite{nicol2021large, bunimovich2023maximal, su2022vector}.

\begin{lem}\label{lem:ml_in_lp}  \phantom{This text will be invisible}
	
	\begin{itemize}
		\item[(i)] 	If $\varphi : [0,1] \to \bR$ is bounded, then 
		for $a \ge a_0(\gamma)$, 
		$$
		\Vert \widetilde{\cL}_\gamma^n \varphi  
		\Vert_{L^\infty(\widetilde{m})} \le 
		a 
		\Vert \varphi \Vert_{L^\infty(\widetilde{m})} 
		\quad \forall n \ge 0.
		$$
		\item[(ii)] Let $\ve > 0$ and $p > 1$. Then, the following holds for any 
		sufficiently small $\delta_* = \delta_*(\gamma, \ve, p) > 0$, and any 
		 $\gamma_0 = (\alpha_0, \beta_0) 
		\in B$, $\gamma_1 = (\alpha_1, \beta_1) \in B$: 
		 there 
		exists a constant $C$ determined by 
		$a, \gamma_0, p,\ve, \gamma, \delta_*$ such that for all 
		$\phi, \psi \in 
		\cC_a(\gamma_0)$ with $m(\phi) = m(\psi)$
		and $a \ge a_0(\gamma_u)$, 
		\begin{align}\label{eq:ml_lp_fast}
			\Vert \widetilde{\mathcal{L}}_{\gamma_1}^n 
			[g^{-1}( \phi - \psi   ) ]
			\Vert_{L^p( \widetilde{m})}
			\le 
			C ( m(\psi) + 
			m(\phi) ) n^{-  \frac{1}{p + \ve}   \gamma_*},
		\end{align} 
		where 
		$$
		\gamma_* = \tfrac{1}{\alpha_u}(
		\tfrac{1}{\beta_u} - \alpha_0 ).
		$$
		Moreover, there exists a constant 
		$C' > 0$ determined by $a, p,\ve, \gamma, \delta_*$ such that 
		\begin{align}\label{eq:ml_lp_slow}
			\Vert \widetilde{\mathcal{L}}_{\gamma_1}^n 
			[g^{-1}( \phi - \psi   ) ]
			\Vert_{L^p( \widetilde{m})}
			\le 
			C' ( m(\psi) + 
			m(\phi) ) n^{  \frac{1}{p + \ve}
				( 1 - \frac{1}{\alpha_u\beta_u} ) }.
		\end{align}
	\end{itemize}
\end{lem}

\begin{proof} (i) If $\varphi : [0,1] \to \bR$ 
	is bounded and $x \in (0,1]$,
	\begin{align*}
		| \widetilde{\cL}_\gamma^n \varphi(x) |
		= | g^{-1} \cL_\gamma^n( g \varphi)(x) |
		\le  \Vert \varphi 
		\Vert_{L^\infty(\widetilde{m})}
		g^{-1} \cL_\gamma^n( g )(x).
	\end{align*}
	Note that $g \in \cC_a(\gamma)$ 
	for $a \ge a_0(\gamma_u)$. 
	It 
	follows by Lemma 
	\ref{lem:inv_of_lsv_cone} that 
	$\cL_\gamma^n( g ) \in \cC_a(\gamma)$. Thus, 
	$$
	| \widetilde{\cL}_\gamma^n \varphi(x) | 
	\le \Vert \varphi \Vert_{L^\infty(\widetilde{m})} 
	g^{-1}(x) \cdot a x^{\frac{1}{\beta} - 
		\alpha - 1}
	\le a \Vert \varphi \Vert_{L^\infty(\widetilde{m})}.
	$$
	
\noindent (ii) Since $\phi, \psi \in \cC_a(\gamma_0) \subset \cC_a(\gamma_u)$,
\begin{align*}
	| \widetilde{\mathcal{L}}_{\gamma_1}^n 
	[g^{-1}( \phi - \psi   ) ](x) | 
	\le g^{-1}(x) ( | \widetilde{\mathcal{L}}_{\gamma_1}^n(\phi)| +     | \widetilde{\mathcal{L}}_{\gamma_1}^n(\psi)|   ) (x)
	\le 2 a m(\phi ) x^{ - \delta_* (   1 + \frac{1}{ \beta ( \beta + \delta_* ) }  )  }.
\end{align*}
Hence, for any $q \ge 1$, 
\begin{align*}
	\int_0^1	| \widetilde{\mathcal{L}}_{\gamma_1}^n 
	[g^{-1}( \phi - \psi   ) ] |^q  \, d \widetilde{m} 
	\le (2  a m ( \phi ))^q  \int_0^1 x^{ \frac{1}{\beta} - \alpha - 1 - q \delta_*  (   1 + \frac{1}{ \beta ( \beta + \delta_* ) }  )   } \, dx < \infty,
\end{align*}
whenever $\delta_*$ is sufficiently small. Next, interpolation is used to obtain
\begin{align*}
	\Vert \widetilde{\mathcal{L}}_{\gamma_1}^n 
	[g^{-1}( \phi - \psi   ) ] \Vert_{L^p(\widetilde{m})} 
	\le 	\Vert \widetilde{\mathcal{L}}_{\gamma_1}^n 
	[g^{-1}( \phi - \psi   ) ] \Vert_{L^1(\widetilde{m})}^{ \frac{1}{p + \ve }  } 
		\Vert \widetilde{\mathcal{L}}_{\gamma_1}^n 
	[g^{-1}( \phi - \psi   ) ] \Vert_{L^q(\widetilde{m})}^{1 -  \frac{1}{p + \ve }  },
\end{align*}
where 
$
\tfrac1p = \tfrac{1}{p + \ve} + \tfrac{ 1 - 1/(p + \ve) }{q}
$. For sufficiently small $\delta_*$ depending on $\gamma, \ve, p$, we have 
\begin{align}\label{eq:lp_l1}
\Vert \widetilde{\mathcal{L}}_{\gamma_1}^n 
[g^{-1}( \phi - \psi   ) ] \Vert_{L^p(\widetilde{m})} 
&\le  ( a m(\phi) )^{ 1 - \frac{1}{p + \ve} } C	\Vert \widetilde{\mathcal{L}}_{\gamma_1}^n 
[g^{-1}( \phi - \psi   ) ] \Vert_{L^1(\widetilde{m})}^{ \frac{1}{p + \ve }  } \notag \\ 
&\le  ( a m(\phi) )^{ 1 - \frac{1}{p + \ve} } C	\Vert \cL_{\gamma_1}^n 
 ( \phi - \psi    ) \Vert_{L^1(m)}^{ \frac{1}{p + \ve }  },
\end{align}
for some constant $C = C(p,\ve, \gamma, \delta_*) > 0$. Hence, by \eqref{eq:ml_strong},
\begin{align*}
	\Vert \widetilde{\mathcal{L}}_{\gamma_1}^n 
	[g^{-1}( \phi - \psi   ) ] \Vert_{L^p(\widetilde{m})}  \le 
	C' m(\phi)  n^{ - \frac{1}{p + \ve}    \frac{1}{\alpha_u} (  \frac{1}{\beta_u} - \alpha_0  )  },
\end{align*}
for some constant $C = C(a, \gamma_0, p,\ve, \gamma, \delta_*) > 0$. This proves \eqref{eq:ml_lp_fast}.  
Similarly we obtain  \eqref{eq:ml_lp_slow} by applying \eqref{eq:ml_weak} instead of \eqref{eq:ml_strong}.
	
\end{proof}

\subsection{Proof of Theorem 
\ref{thm:main} for $L^q$-observables}
Let $\gamma \in \mathfrak{D}$, and 
let $\phi : [0,1] \to \bR$ be a measurable 
    function such that 
    \begin{align*}
    \int_0^1 | \phi(x) |^q x^{\frac{1}{\beta} - 
    \alpha
    - 1} \, dx < \infty
    \end{align*}
    holds for some 
    $q > q_0 := \tfrac{1}{
        1 - \alpha \beta}$. This 
        means that $\phi \in L^q( \widetilde{m} )$ 
        if we 
        set $g(x) = x^{ \frac{1}{\beta} - 
        \alpha- 1 }$ in the definition of 
        $\widetilde{m}$.
In this situation, Theorem \ref{thm:main} can be proved 
by an argument similar to the one for 
$\phi \in L^\infty(m)$, but applying Lemma 
\ref{lem:ml_in_lp} together with 
H\"{o}lder's inequality instead of 
Lemma \ref{lem:memory_loss}. We 
provide details below about what has to be changed 
from the proof for $\phi \in L^\infty(m)$. 

Let $\gamma_u = (\alpha_u, \beta _u) = \gamma + (\delta_*, \delta_*) \in \fD$ for a small but fixed $\delta_* > 0$, and 
set $B = [\alpha_\ell, \alpha_u] \times [1, \beta_u] = [\alpha - \delta_*, \alpha + \delta_*] \times [1, \beta + \delta_*]$.
Let $p$ denote the H\"{o}lder conjugate of $q$. 
Note that  
$$
q > q_0 \iff p < \tfrac{1}{\alpha\beta} 
\iff 1  < \tfrac{1}{p}\tfrac{1}{\alpha\beta}.
$$
First, we verify that $D_i$ is still well-defined 
for $i \in \{1,2\}$. By a change of measure 
followed by an application of 
H\"{o}lder's inequality,  we see that 
\begin{align*}
    \sum_{k=0}^{\infty} \int_0^1
    | \phi \cdot \cL_{\gamma}^k[ ( X_{\gamma, i} 
    \cN_{\gamma,i}
    (h_{\gamma})  )'  ] | \, dm 
    &= \sum_{k=0}^{\infty} \int_0^1
    | \phi \cdot 
    \widetilde{\cL}_{\gamma}^k[ g^{-1} 
        ( X_{\gamma, i} \cN_{\gamma,i}
    (h_{\gamma})  )'  ] | \, d\widetilde{m} \\ 
    &\le 
    \sum_{k=0}^\infty
    \Vert \phi \Vert_{L^q( \widetilde{m} )}
    \Vert
    \widetilde{\cL}_{\gamma}^k[ g^{-1} 
        ( X_{\gamma, i} \cN_{\gamma,i}
    (h_{\gamma})  )'  ] \Vert_{L^p( 
        \widetilde{m} )}.
  \end{align*}
By Lemma \ref{lem:D_well_defined}, 
$( X_{\gamma, i} \cN_{\gamma, i}
(h_{\gamma})  )'$ can be written as a 
difference of two functions in $\cC_a( \delta_* , \beta)$.
Making $\delta_*$ smaller if necessary, Lemma \ref{lem:ml_in_lp}  allows 
us to derive the estimate 
$$
\sum_{k=0}^{\infty} \int_0^1
    | \phi \cdot \cL_{\gamma}^k[ ( X_{\gamma, i} 
    \cN_{\gamma,i}
    (h_{\gamma})  )'  ] | \, dm \le C
\Vert \phi \Vert_{L^q( \widetilde{m} )}
\sum_{k=0}^\infty 
k^{ - \frac{1}{p + \ve } \gamma_* },
$$
where 
$
\gamma_* = \tfrac{1}{\alpha_u}( \tfrac{1}{\beta_u} - \delta_* )
$
and $\ve > 0$ can be chosen to be arbitrarily small. Since $1  < \tfrac{1}{p}\tfrac{1}{\alpha\beta}$, 
we can guarantee $\sum_{k=0}^\infty 
k^{ -  \frac{1}{p + \ve} \gamma_* } < \infty$ by  choosing $\ve$ and $\delta_*$
sufficiently small, and so it follows that $D_i$ is absolutely 
summable for $i \in \{1,2\}$.

%Next, define $B = [\alpha_\ell, \alpha_u] 
%\times [1, \beta_u]$ for some 
%$0 < \alpha_\ell < \alpha$, and 
%let $\delta$  be as in Section \ref{sec:proof_main}.
Suppose that $\Vert \delta \Vert  \le \delta_*$.
Again, by a 
change of measure followed by applications 
of H\"{o}lder's inequality and 
Lemma \ref{lem:ml_in_lp}, we see that for arbitrarily small $\ve$, 
\begin{align*}
  \cR_\phi(\gamma + \delta) -  
  \cR_\phi(\gamma )
  = \int_0^1 \phi \circ T^k_{\gamma + \delta} \, dm 
  - \int_0^1 \phi \circ T^k_\gamma \, dm
  + O ( \Vert \phi  
  \Vert_{L^q(\widetilde{m})} 
  k^{ \frac{1}{p + \ve } ( 1- \frac{1}{\alpha_u 
  \beta_u}  ) })
\end{align*}
for all $k \ge 1$, 
where the constant in the error term 
is determined by $p,\ve, \gamma, \delta_*$ . Let $\xi > 0$. 
Then there exists a constant 
$\mathfrak{c} = \mathfrak{c}(p,\ve, \gamma, \delta_*) > 0$, 
such that for 
$$
k \ge  k_0(\xi, \delta, B,  p) 
:= \lceil \mathfrak{c} \Vert \delta \Vert^{- (1 + \xi) (p + \ve )
\frac{\alpha_u\beta_u}{1 - \alpha_u\beta_u}} \rceil
$$
we have 
\begin{align*}
  \frac{ \cR_\phi(\gamma + \delta) -  
  \cR_\phi(\gamma ) - D \cdot \delta}{ \Vert \delta  \Vert}
  = P_k + O(  \Vert \phi  
  \Vert_{L^q(\widetilde{m})} 
  \Vert \delta  \Vert^\xi ),
\end{align*}
where $P_k$ is as in \eqref{eq:p_k}.
Now, to obtain 
$
\cR_\phi(\gamma + \delta) -  
\cR_\phi(\gamma ) - D \cdot \delta
= o( \Vert \delta \Vert )
$ as $\delta \to 0$, we only 
need to show that the terms 
$Q_k^{(i)}$ and $R_k^{(i)}$ ($i=1,2$)
in the decomposition 
\eqref{eq:decomp_pk} vanish as $\delta \to 0$.

\subsubsection{Limit of $R_k^{(i)}$} Recall that 
\begin{align*}
&R_k^{(i)}
  = \frac{
    \sum_{j=0}^{k-1} \int_{0}^1 \int_{0}^1
    \phi \cdot \cL_{\gamma + \delta}^j \partial_i^2 \cL_{\gamma + \delta t}(\varphi_j)
     \, dm \cdot (1 - t) \delta_i^2 \, dt  
  }{ \Vert \delta \Vert}.
\end{align*}
By \eqref{eq:decomp_rk}, if $\delta_*$ 
is sufficiently small, there exists a decomposition 
$$
\partial_i^2 \cL_{\gamma + \delta }(\varphi_j)
= (\psi_1^{(\delta)} - m(\psi_1^{(\delta)}) - (\psi_2^{(\delta)} - m(\psi_2^{(\delta)})) , 
$$
where $\psi_i^{(\delta)} \in 
\cC_{a}(  \delta_* , \beta_u )$ 
and $\Vert \psi_i^{(\delta)} \Vert_{L^1(m)} 
\le C_i(a, b_1, b_2, b_3, \gamma, \delta_*) < \infty$ 
holds for all $\delta$ 
with $\Vert  \delta \Vert  \le \delta_*$. 
Then, a change of measure, combined with applications 
of H\"{o}lder's inequality and 
Lemma \ref{lem:ml_in_lp}, yields
\begin{align*}
    &\int_{0}^1
    |\phi \cdot \cL_{\gamma + \delta}^j \partial_i^2 \cL_{\gamma + \delta t}(\varphi_j)
     | \, dm 
     = \int_{0}^1
     |\phi \cdot  \widetilde{\cL}_{\gamma + 
     \delta}^j  g^{-1}
     \partial_i^2 \cL_{\gamma + \delta t}(\varphi_j)
      | \, d \widetilde{m} \\
    &\le 
    \Vert \phi \Vert_{L^q( \widetilde{m} )}
    \Vert
    \widetilde{\cL}_{\gamma + 
    \delta}^j  g^{-1}
    \partial_i^2 \cL_{\gamma + \delta t}(\varphi_j)
    \Vert_{L^p( \widetilde{m} )} 
    \le C
    \Vert \phi \Vert_{L^q( \widetilde{m} )}
    j^{ - \frac{1}{p + \ve} \gamma_* },
\end{align*}
where $C = C(a, b_1, b_2, b_3, \gamma, p, \ve, \delta_*) > 0$ is a 
constant and 
$$
\gamma_* =  \tfrac{1}{\alpha_u}(
        \tfrac{1}{\beta_u} - \delta_* ).
$$
Since $\tfrac{1}{p} \tfrac{1}{\alpha\beta} > 1$, 
for sufficiently small $\ve$ and $\delta_*$ we have $\tfrac{1}{p + \ve } \tfrac{1}{\alpha_u}(
\tfrac{1}{\beta_u} - \delta_* ) > 1$ so that the obtained upper bound is summable with respect to $j$. 
Hence, for $\phi \in L^q( \widetilde{m} )$ we still have 
$\lim_{\delta \to 0 } R_k^{(i)} = 0$. 

\subsubsection{Limit of $Q_k^{(i)}$} We want to show that 
\begin{align*}
    |Q_k^{(i)}|
    &\le \biggl| \sum_{j=0}^{k-1}  \int_0^1 \phi \cdot 
    \cL_{\gamma + \delta}^j \partial_i \cL_\gamma
     ( \varphi_j ) \, dm - D_i  \biggr| \to 0, 
  \end{align*}
as $\delta \to 0$, where $
\varphi_j = \cL_{\gamma}^{k-j}( \mathbf{1} )
$. 

%\noindent\textbf{Case $i=1$:} 
As in the case of $D_i$ above, we 
we see that 
$\sum_{j=0}^{k-1}  \int_0^1 \phi \cdot 
\cL_{\gamma + \delta}^j \partial_i \cL_\gamma
 ( \varphi_j ) \, dm$ is absolutely summable.
It follows that, for $\widetilde{\ve} > 0$, there exists 
$N = N( \widetilde{\ve} ) > 0$ such that  
\begin{align*}
    &\biggl| \sum_{j=0}^{k-1}  \int_0^1 \phi \cdot 
    \cL_{\gamma + \delta}^j \partial_i \cL_\gamma
     ( \varphi_j ) \, dm - D_i \biggr| \\
    &\le \biggl| \sum_{j=0}^{N} \int 
    \phi \cdot  ( \widetilde{\cL}_{\gamma + \delta}^j 
    - \widetilde{\cL}_\gamma^j  ) 
    [ g^{-1}
      (X_{\gamma,i} \cN_{\gamma,i}(\varphi_j))'
    ] \, d\widetilde{m}  \biggr|
      \\
    &+ \biggl| \sum_{j=0}^{N} 
    \int_0^1
    \phi \cdot \widetilde{\cL}_{\gamma}^j[ 
        g^{-1}
        (X_{\gamma,i} \cN_{\gamma,i}(\varphi_j))'
        -   
    g^{-1}( X_{\gamma, i} 
    \cN_{\gamma,i}
    (h_{\gamma})  )'  ] \, d\widetilde{m} 
    \biggr| + \tfrac{\widetilde{\ve}}{2},
  \end{align*}
whenever $\Vert \delta \Vert$  is sufficiently small, i.e. 
$k$ is sufficiently large.
Hence, for 
$\lim_{\delta \to 0} Q_k^{(i)} = 0$ it suffices 
to establish the following two properties 
(for a redefined $\widetilde{\ve}$ depending on $\Vert \phi \Vert_{L^q(\widetilde{m})}$ 
and $N$):
\begin{align}\label{eq:to_show_last_1_lq}
    \Vert ( \widetilde{\cL}_{\gamma + \delta}^j 
    - \widetilde{\cL}_\gamma^j  ) 
    [ g^{-1}
      (X_{\gamma,i} \cN_{\gamma,i}(\varphi_j))'
     ]
    \Vert_{L^p( \widetilde{m})}
    \le \widetilde{\ve}
    \quad  \forall  \, 0 \le j \le N,
  \end{align}
  and 
  \begin{align}\label{eq:to_show_last_2_lq}
    \Vert
    \widetilde{\cL}_{\gamma}^j[
   g^{-1} (X_{\gamma,i} \cN_{\gamma,i}(\varphi_j))'
    -   
    g^{-1} ( X_{\gamma,i} 
    \cN_{\gamma,i}
    (h_{\gamma})  )' ]
    \Vert_{L^p( \widetilde{m} )} \le \widetilde{\ve}    
    \quad  \forall  \, 0 \le j \le N.
  \end{align}
  %%%%%%%%%%%%%%%%%%%%%%%%%%%%%%%%%%%%
%%%%%%%%%%%%%%    References    %%%%%%%%%%%%%
%%%%%%%%%%%%%%%%%%%%%%%%%%%%%%%%%%%%
For
convenience, we denote $\psi = 
(X_{\gamma,i} \cN_{\gamma,i}(\varphi_j))'$ and 
$\widehat{\psi} = 
(X_{\gamma,i} \cN_{\gamma,i}(h_\gamma))'$.
For \eqref{eq:to_show_last_1_lq}, 
since 
$\psi$ 
can be written as a difference of two functions in 
$C_a(\gamma)$, we can argue as in \eqref{eq:lp_l1} 
to obtain 
\begin{align*}
&\Vert ( \widetilde{\cL}_{\gamma + \delta}^j 
    - \widetilde{\cL}_\gamma^j  ) 
     g^{-1}
      \psi
    \Vert_{L^p( \widetilde{m})} 
    \le C
    \Vert ( \widetilde{\cL}_{\gamma + \delta}^j 
    - \widetilde{\cL}_\gamma^j  ) 
    g^{-1}
    \psi
    \Vert_{L^1( \widetilde{m})}^{\frac{1}{p + \ve }}  
    = C
    \Vert ( \cL_{\gamma + \delta}^j 
    - \cL_\gamma^j  ) 
    \psi
    \Vert_{L^1(m)}^{\frac{1}{p + \ve }}  
\end{align*}
for some constant $C$ and $\ve > 0$
whenever $\Vert \delta \Vert \le \delta_*$ is sufficiently small. Hence, 
\eqref{eq:to_show_last_1_lq} follows by 
\eqref{eq:to_show_last_1}.

For \eqref{eq:to_show_last_2_lq}, since 
$g^{-1} \psi$ is and $g^{-1} \widehat{\psi}$ 
are bounded by Lemma \ref{lem:D_well_defined}, we can use item (i)
of
Lemma \ref{lem:ml_in_lp} to derive 
the upper bound 
\begin{align*}
\Vert
    \widetilde{\cL}_{\gamma}^j(
   g^{-1} \psi
    -   
    g^{-1} \widehat{\psi} \: )
    \Vert_{L^p( \widetilde{m} )}
&\le \Vert
\widetilde{\cL}_{\gamma}^j(
g^{-1} \psi
-   
g^{-1} \widehat{\psi} \: )
\Vert_{L^\infty( \widetilde{m} )}^{1 - \frac1p} 
\Vert
    \widetilde{\cL}_{\gamma}^j(
   g^{-1} \psi
    -   
    g^{-1} \widehat{\psi} \: )
    \Vert_{L^1( \widetilde{m} )}^{\frac1p} \\
&\le C\Vert
\widetilde{\cL}_{\gamma}^j(
g^{-1} \psi
-   
g^{-1} \widehat{\psi} \: )
\Vert_{L^1( \widetilde{m} )}^{\frac1p} \\ 
&= C\Vert
\cL_{\gamma}^j(
 \psi
-   
\widehat{\psi} \: )
\Vert_{L^1( m )}^{\frac1p} 
\le 
C\Vert
 \psi
-   
\widehat{\psi} 
\Vert_{L^1( m )}^{\frac1p},
\end{align*}
where the $L^1$ contraction property of 
$\cL_\gamma$ was used, and $C > 0$ 
is a constant determined by $a, b_1, \gamma, p$.
Now, \eqref{eq:to_show_last_2_lq} follows by 
\eqref{eq:to_show_last_2}.
%\vskip 1cm
%\medskip

\newpage

\bigskip
\bigskip
\bibliography{intermittent-critical}{}
\bibliographystyle{plainurl}

%%%%%%%%%%%%%%%%%%%%%%%%%%%%%%%%%%%%

\vspace*{\fill}

\end{document}